\documentclass[11pt,reqno]{amsart}
\usepackage{amsmath,amssymb,amsthm,enumerate,natbib,color,bbm,ifthen,graphicx,overpic}
\usepackage[latin1]{inputenc} 
\def\nset{{\mathbb{N}}}
\def\rset{\mathbb R}

\def\eqsp{\;}

\newcommand{\as}{\text{a.s.} }
\newcommand{\pscal}[2]{\langle #1, #2 \rangle}
\newcommand{\un}{\ensuremath{\mathbbm{1}}}
\newcommand{\eqdef}{\ensuremath{\stackrel{\mathrm{def}}{=}}}

\def\Xset{\mathsf{X}} 
\def\Xsigma{\mathcal{X}} 
\def\Tsigma{\mathcal{B}(\Theta)} 
\def\F{\mathcal{F}} 

\def\rate{\mathbf{r}}

\def\PP{\mathbb{P}} 
\def\PE{\mathbb{E}} 
\def\bPP{\overline{\mathbb{P}}} 
\def\bPE{\overline{\mathbb{E}}} 

\def\L{\mathcal{L}} 

\def\tv{\mathrm{TV}}
\def\s{\mathrm{s}}
\def\compact{\mathsf{K}}  
\def\Cset{\mathcal{C}} 
\def\Dset{\mathcal{D}} 

\newlength{\noteWidth}
\newtheorem{theo}{Theorem}[section]
\newtheorem{lemma}[theo]{Lemma}
\newtheorem{coro}[theo]{Corollary}
\newtheorem{prop}[theo]{Proposition}

\newtheorem{algo}{Algorithm}[section]
\theoremstyle{remark}
\newtheorem{rem}{Remark}
\newtheorem{example}{Example}

\newcounter{hypoconbis}
\newcounter{saveconbis}
\newcommand\debutA{\begin{list} {\textbf{A\arabic{hypoconbis}}}{\usecounter{hypoconbis}}\setcounter{hypoconbis}{\value{saveconbis}}}
\newcommand\finA{\end{list}\setcounter{saveconbis}{\value{hypoconbis}}}

\newcounter{hypoconbisp}
\newcounter{saveconbisp}
\newcommand\debutAp{\begin{list} {\textbf{A\arabic{hypoconbisp}'}}{\usecounter{hypoconbisp}}\setcounter{hypoconbisp}{\value{saveconbisp}}}
\newcommand\finAp{\end{list}\setcounter{saveconbisp}{\value{hypoconbisp}}}

\newcounter{hypocom}
\newcounter{savecom}
\newcommand{\debutB}{\begin{list}{\textbf{B\arabic{hypocom}}}{\usecounter{hypocom}}\setcounter{hypocom}{\value{savecom}}}
\newcommand{\finB}{\end{list}\setcounter{savecom}{\value{hypocom}}}

\newcounter{hypocomp}
\newcounter{savecomp}
\newcommand{\debutBp}{\begin{list}{\textbf{B\arabic{hypocomp}'}}{\usecounter{hypocomp}}\setcounter{hypocomp}{\value{savecomp}}}
\newcommand{\finBp}{\end{list}\setcounter{savecomp}{\value{hypocomp}}}

\newcounter{hypostab}
\newcounter{savestab}
\newcommand{\debutC}{\begin{list}{\textbf{C\arabic{hypostab}}}{\usecounter{hypostab}}\setcounter{hypostab}{\value{savestab}}}
\newcommand{\finC}{\end{list}\setcounter{savestab}{\value{hypostab}}}

\newcounter{hypodist}
\newcounter{savedist}
\newcommand{\debutD}{\begin{list}{\textbf{D\arabic{hypodist}}}{\usecounter{hypodist}}\setcounter{hypodist}{\value{savedist}}}
\newcommand{\finD}{\end{list}\setcounter{savedist}{\value{hypodist}}}

\begin{document}

\title[Subgeometric Adaptive MCMC]{Limit theorems for some adaptive MCMC
  algorithms with subgeometric kernels
}

\author[Y. Atchadé]{Yves Atchadé}

\thanks{ Y. Atchadé: University of Michigan, 1085 South University, Ann Arbor,
  48109, MI, United States. {\em E-mail address} yvesa@umich.edu}

\author[G. Fort]{ Gersende Fort} \thanks{G. Fort: LTCI, CNRS-TELECOM ParisTech,
  46 rue Barrault, 75634 Paris Cedex 13, France. {\em E-mail address}
  gfort@tsi.enst.fr}

\thanks{This work is partly supported by the french National Research Agency
  (ANR) under the program ANR-05-BLAN-0299.}

\subjclass[2000]{60J10, 65C05}

\keywords{Adaptive Markov chain Monte Carlo, Markov chain, Subgeometric
ergodicity.}

\maketitle


\begin{abstract}
This paper deals with the ergodicity (convergence of the marginals) and the law of large numbers for adaptive MCMC algorithms built from transition kernels that are not necessarily geometrically ergodic.
We develop a number of results that broaden significantly the class of adaptive
MCMC algorithms for which rigorous analysis is now possible. As an example, we
give a detailed analysis of the Adaptive Metropolis Algorithm of
\cite{haarioetal00} when the target distribution is sub-exponential in the
tails.
\end{abstract}

\bigskip

\setcounter{secnumdepth}{3}


\section{Introduction}
  This paper deals with the convergence of Adaptive Markov Chain Monte Carlo (AMCMC).
  Markov Chain Monte Carlo (MCMC) is a well known, widely used
method to sample from arbitrary probability distributions. One of the major
limitation of the method is the difficulty in finding sensible values for the
parameters of the Markov kernels.  Adaptive MCMC provides a general
framework to tackle this problem where the parameters are adaptively tuned,
often using previously generated samples.  This approach generates a class of
stochastic processes that is the object of this paper.

Denote $\pi$ the probability measure of interest on some measure space
$(\Xset,\mathcal{X})$. Let $\{P_\theta,\theta\in\Theta\}$ be a family of
$\phi$-irreducible and aperiodic Markov kernels each with invariant
distribution $\pi$. We are interested in the class of stochastic processes
based on non-homogeneous Markov chains $\{(X_n,\theta_n),\;n\geq 0\}$ with
transition kernels $\{\bar P\left(n; (x,\theta); (dx',d\theta')\right), n\geq 0
\}$ satisfying $\int_{\Theta} \bar P\left(n; (x,\theta); (\cdot,d\theta')
\right) = P_\theta(x,\cdot)$. Often, these transition kernels are of the form
$\{P_\theta(x,dy)\delta_{H_{n}(\theta,y)}(d\theta'), n\geq 0\}$ where
$\{H_l,\;l\geq 0\}$ is a family measurable functions, $H_l:\; \Theta\times
\Xset\to \Theta$. The stochastic approximation dynamic corresponds to the case
$H_l(\theta,x)=\theta+\gamma_l \; H(\theta,x)$. In this latter case, it is
assumed that the best values for $\theta$ are the solutions of the equation
$\int H(\theta,x)\pi(dx)=0$. Since the pioneer work of \cite{gilksetal98,
  holden98, haarioetal00, andrieuetrobert02}, the number of AMCMC algorithms in
the literature has significantly increased in recent years.  But despite many
recent works on the topic, the asymptotic behavior of these algorithms is still
not completely understood. Almost all previous works on the convergence of
AMCMC are limited to the case when each kernel $P_\theta$ is geometrically
ergodic (see e.g..  \cite{rosenthaletroberts05,andrieuetal06}). In this paper,
we weaken this condition and consider the case when each transition kernel is
sub-geometrically ergodic.

More specifically, we study the ergodicity of the marginal $\{X_n, n\geq 0 \}$
i.e. the convergence to $\pi$ of the distribution of $X_n$ irrespective of the
initial distribution, and the existence of a strong law of large numbers
for AMCMC.

We first show that a diminishing adaptation assumption of the form
$|\theta_n-\theta_{n-1}|\to 0$ in a sense to be made precise (assumption
B\ref{B1}) together with a uniform-in-$\theta$ positive recurrence towards a
small set $C$ (assumptions A\ref{A-VCset}(\ref{Anew}) and
A\ref{A-VCset}(\ref{A3rev})) and a uniform-in-$\theta$ ergodicity condition of
the kernels $\{P_\theta, \theta \in \Theta\}$ (assumption
A\ref{A-VCset}(\ref{A4rev})) are enough to imply the ergodicity of AMCMC.

We believe that this result is close to be optimal.  Indeed, it is well documented in the literature that
AMCMC can fail to be ergodic if the diminishing assumption does not hold (see
e.g.  \cite{rosenthaletroberts05} for examples). Furthermore, the additional
assumptions are also fairly weak since in the case where $\Theta$ is reduced to
the single point $\{\theta_\star\}$ so that $\{X_n, n\geq 0\}$ is a Markov
chain with transition kernel $P_{\theta_\star}$, these conditions hold if
$P_{\theta_\star}$ is an aperiodic positive that is polynomially ergodic.

We then prove a strong law of large numbers for AMCMC. We show that the
diminishing adaptation assumption and a uniform-in-$\theta$ polynomial drift
condition towards a small set $\Cset$ of the form $P_\theta V\leq V-c
V^{1-\alpha}+b\un_{\Cset}(x)$, $\alpha\in (0,1)$, implies a strong law of large
number for all real-valued measurable functions $f$ for which
$\sup_{\Xset}(|f|/V^{\beta})<\infty$, $\beta\in[0,1-\alpha)$. This result is
close to what can be achieved with Markov chains (with fixed transition kernel)
under similar conditions (\cite{meynettweedie93}).

On a more technical note, this paper makes two key contributions to the
analysis of AMCMC. Firstly, to study the ergodicity, we use a more careful
coupling technique which extends the coupling approach of
\cite{rosenthaletroberts05}. Secondly, we tackle the law of large numbers using
a resolvent kernel approach together with martingales theory. This approach has
a decisive advantage over the more classical Poisson equation approach
(\cite{andrieuetal06}) in that no continuity property of the resolvent kernels
is required. It is also worth noting that the results developed in this paper
can be applied to adaptive Markov chains beyond Markov Chain Monte Carlo
simulation provided all the transition kernels have the same invariant
distribution.

The remainder of the paper is organized as follows. In Section
\ref{sec:ResultsUnif} we state our assumptions followed by a statement of our
main results. Detailed discussion of the assumptions and some comparison with
the literature are provided in Section \ref{sec:discussionUnif}. We apply our
results to the analysis of the Adaptive Random Walk Metropolis algorithm of
\cite{haarioetal00} when the target distribution is sub-exponential in the
tails.  This is covered in Section \ref{sec:Example} together with a toy
example taken from \cite{atchadeetrosenthal03}. All the proofs are postponed to
Section~\ref{sec:Proofs}.

\section{Statement of the results and discussion}\label{sec:ResultsUnif}
\subsection{Notations}\label{sec:notations}
For a transition kernel $P$ on a measurable general state space
$(\mathbb{T},\mathcal{B}(\mathbb{T}))$, denote by $P^n$, $n\geq 0$, its $n$-th
iterate defined as
\[
P^0(x,A) \eqdef \delta_x(A) \eqsp, \qquad \qquad P^{n+1}(x,A) \eqdef \int
P(x,dy ) P^n(y,A) \eqsp, \quad n \geq 0 \eqsp;
\]
$\delta_x(dt)$ stands for the Dirac mass at $\{x\}$. $P^n$ is a transition kernel
on $(\mathbb{T},\mathcal{B}(\mathbb{T}))$ that acts both on bounded measurable
functions $f$ on $\mathbb{T}$ and on $\sigma$-finite measures $\mu$ on
$(\mathbb{T},\mathcal{B}(\mathbb{T}))$ via $P^nf(\cdot) \eqdef \int
P^n(\cdot,dy) f(y)$ and $\mu P^n(\cdot) \eqdef \int \mu(dx) P^n(x, \cdot)$.

If $V: \mathbb{T}\to [1, +\infty)$ is a function, the $V$-norm of a function
$f: \mathbb{T}\to \rset$ is defined as $|f|_V \eqdef \sup_{\mathbb{T}} |f|
/V$.  When $V=1$, this is the supremum norm. The set of functions with finite
$V$-norm is denoted by $\L_V$.

If $\mu$ is a signed measure on a measurable space
$(\mathbb{T},\mathcal{B}(\mathbb{T}))$, the total variation norm $\| \mu
\|_{\tv}$ is defined as
\[
\| \mu \|_{\tv} \eqdef \sup_{\{f, |f|_1 \leq 1 \}} | \mu(f)| = 2 \; \sup_{A \in
  \mathcal{B}(\mathbb{T})}|\mu(A)|= \sup_{A \in \mathcal{B}(\mathbb{T})} \mu(A)
- \inf_{A \in \mathcal{B}(\mathbb{T})} \mu(A) \eqsp;
\]
and the $V$-norm, for some function $V : \mathbb{T} \to [1, +\infty)$, is
defined as $\| \mu \|_{V} \eqdef \sup_{\{g, |g|_V \leq 1 \}} |\mu(g)|$.

\bigskip

Let $\Xset, \Theta$ be two general state space resp. endowed with a countably
generated $\sigma$-field $\Xsigma$ and $\Tsigma$. Let $\{P_\theta, \theta \in
\Theta \}$ be a family of Markov transition kernels on $(\Xset,\Xsigma)$ such
that for any $(x,A) \in \Xset \times \Xsigma$, $\theta \mapsto P_\theta(x,A)$
is measurable.  Let $\{\bar P(n;\cdot,\cdot), n \geq 0 \}$ be a family of transition kernels on $(\Xset \times \Theta, \Xsigma \otimes \Tsigma)$,
satisfying for any $A \in \Xsigma$,
\begin{equation}\label{eq:tk1}
\int_{A \times \Theta} \bar P\left(n; (x,\theta); (dx',d\theta')\right)  =  P_{\theta}(x, A) \eqsp.
\end{equation}
An adaptive Markov chain is a non-homogeneous Markov chain $\{ Z_n =
(X_n,\theta_n), n\geq 0 \}$ on $\Xset\times\Theta$ with transition kernels
$\{\bar P(n; \cdot; \cdot), n \geq 0\}$. 

Among examples of such transition kernels, consider the case when
$\{(X_n,\theta_n), n\geq 0\}$ is obtained through the algorithm: given
$(X_n,\theta_n)$, sample $X_{n+1} \sim P_{\theta_n}(X_n, \cdot)$ and set
$\theta_{n+1} = \theta_n$ with probability $1-p_{n+1}$ or set $\theta_{n+1}=
\tilde \Xi_{n+1}(X_n,\theta_n,X_{n+1})$ with probability $p_{n+1}$. Then
\begin{multline*}
  \bar P\left(n; (x,\theta); (dx',d\theta')\right) = P_\theta(x,dx') \ \left\{
    \left(1-p_{n+1} \right) \  \delta_\theta(d\theta') + p_{n+1} \ \delta_{\tilde
      \Xi_{n+1}(x,\theta,x')}(d\theta') \right\} \eqsp.
\end{multline*}
A special case is the case when $p_{n+1}=1$ and $\theta_{n+1} =
H_{n+1}(\theta_n,X_{n+1})$, where $\{H_l, l\geq 0 \}$ is a family of measurable
functions $H_l: \Theta \times \Xset \to \Theta$.  Then,
\[
\bar P\left(n; (x,\theta); (dx',d\theta')\right) \eqdef P_{\theta}(x, dx') \
\ \delta_{H_{n+1}(\theta,x')}(d \theta') \eqsp.
\]
Such a situation occurs for example if $\theta_{n+1}$ is updated following a
stochastic approximation dynamic: $\theta_{n+1} = \theta_n + \gamma_{n+1}
H(\theta_n,X_{n+1})$.

From $\{\bar P\left(n;\cdot,\cdot\right),\;n\geq 0\}$ and for any integer
$l\geq 0$, we introduce a family - indexed by $l$ - of sequence of transition
kernels $\{\bar P_l(n;\cdot,\cdot), n \geq 0 \}$, where $\bar
P_l\left(n;\cdot,\cdot\right) \eqdef \bar P\left(l+n;\cdot,\cdot\right)$ and we
denote by $\PP_{x,\theta}^{(l)}$ and $\PE_{x,\theta}^{(l)}$ the probability and
expectation on the canonical space $(\Omega, \F)$ of the canonical
non-homogeneous Markov chain $\{ Z_n = (X_n,\theta_n), n\geq 0 \}$ with
transition kernels $\{\bar P_l(n; \cdot; \cdot), n \geq 0\}$ and initial
distribution $\delta_{(x,\theta)}$.  We denote by $\underline{\theta}$ the
shift operator on $\Omega$ and by $\{\F_k, k \geq 0 \}$ the natural filtration
of the process $\{Z_k, k\geq 0\}$.  We use the notations $\PP_{x,\theta}$ and
$\PE_{x,\theta}$ as shorthand notations for $\PP_{x,\theta}^{(0)}$ and
$\PE_{x,\theta}^{(0)}$.

Set
\[
D(\theta,\theta') \eqdef \sup_{x \in \Xset} \| P_{\theta}(x,\cdot) -
P_{\theta'}(x,\cdot) \|_{\tv} \eqsp.
\]

\subsection{Convergence of the marginals}
We assume that minorization, drift conditions and ergodicity are available for
$P_\theta$ uniformly in $\theta$.  For a set $\Cset$, denote by $\tau_\Cset$
the return-time to $\Cset \times \Theta$ : $\tau_\Cset \eqdef \inf\{n \geq 1,
X_n \in \Cset \}$.

\debutA
\item \label{A-VCset} There exist a measurable function $V: \Xset \to
  [1,+\infty)$ and a measurable set $\Cset$ such that
  \begin{enumerate}[(i)]
  \item \label{Anew} $\sup_l \sup_{\Cset \times \Theta}
    \PE_{x,\theta}^{(l)}\left[\rate(\tau_\Cset) \right] < +\infty$ for some
    non-decreasing function $\rate : \nset \to (0, +\infty)$ such that $\sum_n
    1/\rate(n) < +\infty$.
  \item \label{A4rev} there exist a probability measure $\pi$ such that
\[
\lim_{n \to +\infty} \ \sup_{x \in \Xset} V^{-1}(x) \ \sup_{\theta \in \Theta}
\| P^n_\theta(x, \cdot) -\pi \|_{\tv}  = 0 \eqsp.
\]
\item \label{A3rev} $ \sup_\theta P_\theta V \leq V$ on $\Cset^c$ and $
  \sup_{\Cset \times \Theta} \{P_\theta V(x) + V(x) \} < +\infty$.
\end{enumerate}
\finA
\debutB
\item \label{B1} There exist probability distributions $\xi_1, \xi_2$ resp. on
  $\Xset, \Theta$ such that for any $\epsilon>0$, $ \lim_n
  \PP_{\xi_1,\xi_2}\left( D(\theta_n, \theta_{n-1}) \geq \epsilon \right)=0$.
  \finB
\begin{theo}
\label{theo:MarginalUnifCase}
Assume A\ref{A-VCset} and B\ref{B1}. Then
\[
\lim_{n \to +\infty} \sup_{\{f, |f|_1 \leq 1 \}} \left|
  \PE_{\xi_1, \xi_2}\left[f(X_n) - \pi (f)\right] \right| = 0 \eqsp.
\]
  \end{theo}
  Sufficient conditions for A\ref{A-VCset} to hold are the following
  uniform-in-$\theta$ conditions \debutA
\item \label{Adrift}
  \begin{enumerate}[(i)]
  \item The transition kernels $P_\theta$ are $\phi$-irreducible, aperiodic.
  \item There exist a function $V : \Xset \to [1, +\infty)$, $\alpha \in
  (0,1)$ and constants $b,c$ such that for any $\theta \in \Theta$
\[
P_\theta V(x) \leq V(x) - c\ V^{1-\alpha}(x) + b\un_\Cset(x) \eqsp.
\]
\item For any level set $\Dset$ of $V$, there exist $\epsilon_\Dset>0$ and a
  probability $\nu_\Dset$ such that for any $\theta$, $P_\theta(x, \cdot) \geq
  \epsilon_\Dset \un_\Dset(x) \ \nu_\Dset(\cdot)$.
  \end{enumerate}
  \finA
We thus have the corollary
 \begin{coro}{(of Theorem~\ref{theo:MarginalUnifCase})}
\label{coro:MarginalUnifCase}
   Assume A\ref{Adrift}  and B\ref{B1}. Then
\[
\lim_{n \to +\infty} \sup_{\{f, |f|_1 \leq 1 \}} \left|
  \PE_{\xi_1, \xi_2}\left[f(X_n) - \pi (f)\right] \right| = 0 \eqsp.
\] \end{coro}
 Assumption A\ref{A-VCset}(\ref{Anew}) and A\ref{A-VCset}(\ref{A3rev}) are
 designed to control the behavior of the chain ``far from the center''. When
 the state space $\Xset$ is ``bounded'' so that for example, $V=1$ in
 A\ref{A-VCset}(\ref{A4rev}), then we have the following result
    \begin{lemma}
 \label{lemma:MarginalUnifCaseBounded}
 If there exists a probability measure $\pi$ such that $\lim_{n \to +\infty} \
 \sup_{ \Xset \times \Theta} \| P^n_\theta(x, \cdot) -\pi(\cdot) \|_{\tv} = 0 $, then
 A\ref{A-VCset}(\ref{Anew}) and A\ref{A-VCset}(\ref{A3rev}) hold with a bounded
 function $V$ and $\Cset = \Xset$.
\end{lemma}
Combining the assumptions of Lemma~\ref{lemma:MarginalUnifCaseBounded} and
B\ref{B1}, we deduce from Theorem~\ref{theo:MarginalUnifCase} the convergence
of the marginals. This result coincides with \cite[Theorem
5]{rosenthaletroberts05}. As observed by \cite{Bai:2008} (personal
communication), assumption A\ref{Adrift} also imply the ``containment
condition'' as defined in \cite{rosenthaletroberts05}.  Consequently,
Corollary~\ref{coro:MarginalUnifCase} could also be established by applying
\cite[Theorem 13]{rosenthaletroberts05}: this would yield to the following
statement, which is adapted from \cite{Bai:2008}. Define $M_\epsilon(x,\theta)
\eqdef \inf \{n \geq 1, \|P_\theta^n(x,\cdot) - \pi(\cdot) \|_\tv \leq \epsilon
\}$.
\begin{prop}
\label{prop:YanBai}
Assume A\ref{Adrift} and  B\ref{B1}. Then for any
$\epsilon>0$, the sequence $\{M_\epsilon(X_n,\theta_n), n\geq 0 \}$ is bounded
in probability for the probability $\PP_{\xi_1,\xi_2}$ and
\[
\lim_{n \to +\infty} \sup_{\{f, |f|_1 \leq 1 \}} \left| \PE_{\xi_1,
    \xi_2}\left[f(X_n) - \pi (f)\right] \right| = 0 \eqsp.
\]
\end{prop}

\subsection{Strong law of large numbers}
Assumptions A\ref{A-VCset} and B\ref{B1} are strengthened as follows \debutA
\item \label{A2} There exist a probability measure $\nu$ on $\Xset$, a positive
  constant $\varepsilon$ and a set $\Cset \in \Xsigma$ such that for any
  $\theta \in \Theta$, $P_\theta(x,\cdot) \geq \un_\Cset(x) \ \varepsilon
  \nu(\cdot)$.
\item \label{A5} There exist a measurable function $V: \Xset \to [1,+\infty)$,
  $0 < \alpha < 1$ and positive constants $b,c$ such that for any $\theta \in
  \Theta$, $P_\theta V \leq V - c \ V^{1-\alpha} + b \un_\Cset$.
\item \label{A6} There exist a probability measure $\pi$ and some $0 \leq
  \beta < 1-\alpha$ such that for any level set $\Dset \eqdef \{x \in \Xset, V(x)
  \leq d \}$ of $V$,
\[
\lim_{n \to +\infty} \ \sup_{\Dset \times \Theta} \| P^n_\theta(x, \cdot) -\pi
\|_{V^\beta} = 0 \eqsp.
\]
\finA
\debutB
\item \label{B2} For any level set $\Dset$ of $V$ and any $\epsilon>0$,
\[
\lim_n \sup_{l \geq 0} \sup_{ \Dset \times \Theta} \PP_{x,\theta}^{(l)}\left(
  D(\theta_n, \theta_{n-1}) \geq \epsilon \right)=0 \eqsp.
\]
\finB

\begin{theo}
  \label{theo:SLLNUnboundedUnifCase}
  Assume A\ref{A2}-\ref{A6} and B\ref{B2}. Then for any
  measurable function $f: \Xset \to \rset$ in $\L_{V^\beta}$ and any initial
  distribution $\xi_1,\xi_2$ resp. on $\Xset, \Theta$ such that $\xi_1(V) <
  +\infty$,
\[
\lim_{n \to +\infty} n^{-1} \sum_{k=1}^n f(X_k) = \pi(f) \eqsp, \qquad \qquad
\PP_{\xi_1,\xi_2}-\as
\]
\end{theo}

As in the case of the convergence of the marginals, when A\ref{A6} and
  B\ref{B2} hold with $\Dset = \Xset$ and $\beta = 0$, A\ref{A2} and A\ref{A5}
  can be omitted.  We thus have
\begin{prop}
    \label{prop:SLLNUnboundedUnifCaseBounded}
    Assume that A\ref{A6} and B\ref{B2} hold with $\Dset = \Xset$ and
    $\beta=0$. Then for any measurable bounded function $f: \Xset \to \rset$
    and any initial distribution $\xi_1,\xi_2$ resp. on $\Xset, \Theta$
\[
\lim_{n \to +\infty} n^{-1} \sum_{k=1}^n f(X_k) = \pi(f) \eqsp, \qquad \qquad
\PP_{\xi_1,\xi_2}-\as
\]
\end{prop}

\subsection{Discussion}\label{sec:discussionUnif}
\subsubsection{Non-adaptive case}
We start by comparing our assumptions to assumptions in Markov chain theory
under which the law of large numbers hold. In the setup above, taking
$\Theta=\{\theta_\star\}$ and $H(\theta_\star,x) = \theta_\star$ reduces
$\{X_n, n\geq 0\}$ to a Markov chain with transition kernel $P_{\theta_\star}$.
Assume that $P_{\theta_\star}$ is Harris-recurrent.

In that case, a condition which is known to be minimal and to imply ergodicity
in total variation norm is that $P_{\theta_\star}$ is an aperiodic positive
Harris recurrent transition kernel \cite[Theorems 11.0.1 and
13.0.1]{meynettweedie93}.  Condition A\ref{A-VCset}(\ref{Anew}) is stronger
than positive Harris recurrence since it requires $\sup_\Cset \PE_x
[\rate(\tau_\Cset)]<+\infty$ for some rate $\rate$, $\rate(n)>> n$.
Nevertheless, as discussed in the proof (see remark~\ref{rem:YanBai},
Section~\ref{sec:Proofs}), the condition $\sum_n \{1/\rate(n) \} <+\infty$ is
really designed for the adaptive case.  A\ref{A-VCset}(\ref{A4rev}) is stronger
than what we want to prove (since A\ref{A-VCset}(\ref{A4rev}) implies the
conclusion of Theorem~\ref{theo:MarginalUnifCase} in the non-adaptive case);
this is indeed due to our technique of proof which is based on the comparison
of the adaptive process to a process - namely, a Markov chain with transition
kernel $P_\theta$ - whose stationary distribution is $\pi$. Our proof is thus
designed to address the adaptive case. Finally, B\ref{B1} is trivially true.

For the strong law of large numbers (Theorem \ref{theo:SLLNUnboundedUnifCase}),
B\ref{B2} is still trivially true in the Markovian case and A\ref{A6} is
implied by A\ref{A2} and A\ref{A5} combined with the assumption that
$P_{\theta_\star}$ is $\phi$-irreducible and aperiodic (see
Appendix~\ref{app:UniformControl} and references therein).  In the Markovian
case, whenever $P_{\theta_\star}$ is $\phi$-irreducible and aperiodic,
A\ref{A2} and A\ref{A5} are known sufficient conditions for a strong law of
large numbers for $f \in \L_{V^{1-\alpha}}$, which is a bit stronger than the
conclusions of Theorem~\ref{theo:SLLNUnboundedUnifCase}.  This slight loss of
efficiency is due to the technique of proof based on martingale theory (see
comments Section~\ref{subsec:MethodsProof}).  Observe that in the geometric
case, there is the same loss of generality in \cite[Theorem 8]{andrieuetal06}.
More generally, any proof of the law of large numbers based on the martingale
theory (through for example the use of the Poisson's equation or of the
resolvent kernel) will incur the same loss of efficiency since limit theorems
exist only for $L^p$-martingale with $p>1$.

\subsubsection{Checking assumptions A\ref{A-VCset}(\ref{A4rev}) and A\ref{A6}}
\label{subsec:CheckCond}
A\ref{A-VCset}(\ref{A4rev}) and A\ref{A6} are the most technical of our
assumptions. Contrary to the case of a single kernel, the relations between
A\ref{A-VCset}(\ref{A4rev}) (resp. A\ref{A6}) and
A\ref{A-VCset}(\ref{Anew})-A\ref{A2} (resp. A\ref{A2}, A\ref{A5}) are not
completely well understood. Nevertheless these assumptions can be checked under
conditions which are essentially of the form A\ref{A2}, A\ref{A5} plus the
assumptions that each transition kernel $P_\theta$ is $\phi$-irreducible and
aperiodic, as discussed in Appendix~\ref{app:UniformControl}.

\subsubsection{On the uniformity in $\theta$ in assumptions  A\ref{A-VCset}(\ref{Anew}), A\ref{A-VCset}(\ref{A4rev}), A\ref{A2} and A\ref{A5}}
We have formulated A\ref{A-VCset}(\ref{Anew}), A\ref{A-VCset}(\ref{A4rev}),
A\ref{A2} and A\ref{A5} such that all the constants involved are independent of
$\theta$, for $\theta \in\Theta$.  Intuitively, this corresponds to AMCMC
algorithms based on kernels with overall similar ergodicity properties. This
uniformity assumption might seem unrealistically strong at first. But the next
example shows that when these conditions do not hold
uniformly in $\theta$ for $\theta \in\Theta$, pathologies can occur if the
adaptation parameter can wander to the boundary of $\Theta$.

\begin{example}
  The example is adapted from \cite{winkler03}. Let $\Xset=\{0,1\}$ and
  $\{P_\theta,\;\theta\in(0,1)\}$ be a family of transition matrices with
  $P_\theta(0,0)=P_\theta(1,1)=1-\theta$. Let $\{\theta_n, n\geq 0\}$,
  $\theta_n \in (0,1)$, be a deterministic sequence of real numbers decreasing
  to $0$ and $\{X_n, n\geq 0\}$ be a non-homogeneous Markov chain on $\{0,1\}$
  with transition matrices $\{P_{\theta_n}, n\geq 0\}$. One can check that
  $D(\theta_n,\theta_{n-1})\leq \theta_{n-1}-\theta_n$ for all $n\geq 1$ so
  that B\ref{B1} and B\ref{B2} hold.
  
  For any compact subset $\compact$ of $(0,1)$, it can be checked that
  A\ref{A-VCset}(\ref{Anew}), A\ref{A-VCset}(\ref{A4rev}), A\ref{A2} and
  A\ref{A5} hold uniformly for all $\theta\in\compact$. But these assumptions
  do not hold uniformly for all $\theta\in (0,1)$. Therefore Theorems
  \ref{theo:MarginalUnifCase} and \ref{theo:SLLNUnboundedUnifCase} do not
  apply.  Actually one can easily check that $\PP_{x,\theta_0}\left(X_n\in
    \cdot\right) \to \pi(\cdot)$ as $n\to\infty$, but that
  $\PE_{x,\theta_0}\left[\left(n^{-1}\sum_{k=1}^n
      f(X_k)-\pi(f)\right)^2\right]$ do not converge to $0$ for bounded
  functions $f$. That is, the marginal distribution of $X_n$ converges to $\pi$
  but a weak law of large numbers fails to hold.
\end{example}

This raises the question of how to construct AMCMC when
A\ref{A-VCset}(\ref{Anew}), A\ref{A-VCset}(\ref{A4rev}), A\ref{A2} and
A\ref{A5} do not hold uniformly for all $\theta\in\Theta$. When these
assumptions hold uniformly on any compact subsets of $\Theta$ and the
adaptation is based on stochastic approximation, one approach is to stop the
adaptation or to reproject $\theta_n$ back on $\mathcal{K}$ whenever
$\theta_n\notin\mathcal{K}$ for some fixed compact $\mathcal{K}$ of $\Theta$. A
more elaborate strategy is Chen's truncation method which - roughly speaking -
reinitializes the algorithm with a larger compact, whenever
$\theta_n\notin\mathcal{K}$ (\cite{chen:zhu:1986,chen:gua:gao:1988}).  A third
strategy consists in proving a drift condition on the bivariate process
$\{(X_n,\theta_n), n \geq 0\}$ in order to ensure the stability of the process
(\cite{andrieu:vlad:2008}, see also \cite{benveniste:metivier:priouret:1987}).
This question is however out of the scope of this paper; the use of the Chen's
truncation method to weaken our assumption is addressed in
\cite{atchade:fort:2008b}.

\subsubsection{Comparison with the literature}
\label{subsec:CompLite}
The convergence of AMCMC has been considered in a number of early works, most
under a geometric ergodicity assumption. \cite{haarioetal00} proved the
convergence of the adaptive Random Walk Metropolis (ARWM) when the state space
is bounded. Their results were generalized to unbounded spaces in
\cite{atchadeetrosenthal03} assuming the diminishing adaptation assumption and
a geometric drift condition of the form
\begin{equation}\label{GeoDrift}P_\theta V(x)\leq \lambda V(x)+b\textbf{1}_C(x),\end{equation}
for $\lambda\in (0,1)$, $b<\infty$ and $\theta\in\Theta$.

\cite{andrieuetal06} undertook a thorough analysis of adaptive chains under the
geometric drift condition (\ref{GeoDrift}) and proved a strong law of large
numbers and a central limit theorem. \cite{andrieuetatchade05} gives a theoretical discussion on the efficiency of
AMCMC under (\ref{GeoDrift}).

\cite{rosenthaletroberts05} improves on the literature by relaxing the
convergence rate assumption on the kernels.  They prove the convergence of the
marginal and a weak law of large numbers for bounded functions.  But their
analysis requires a uniform control on certain moments of the drift function, a
condition which is easily checked in the geometric case (i.e. when
A\ref{Adrift} or A\ref{A5} is replaced with (\ref{GeoDrift})).  Till recently,
it was an open question in the polynomial case but this has been recently
solved by \cite{Bai:2008} - contemporaneously with our work - who proves that
such a control holds under conditions which are essentially of the form
A\ref{Adrift}.

\cite{yang:2007} tackles some open questions mentioned in
\cite{rosenthaletroberts05}, by providing sufficient conditions - close to the
conditions we give in Theorems~\ref{theo:MarginalUnifCase} and
\ref{theo:SLLNUnboundedUnifCase} - to ensure convergence of the marginals and a
weak law of large numbers for bounded functions.  The conditions in
\cite[Theorems 3.1 and 3.2]{yang:2007} are stronger than our conditions.  But
we have noted some skips and mistakes in the proofs of these theorems.

 \subsubsection{Comments on the methods of proof}
\label{subsec:MethodsProof}
The proof of Theorem \ref{theo:MarginalUnifCase} is based on an argument
extended from \cite{rosenthaletroberts05} which can be sketched heuristically
as follows. For $N$ large enough, we can expect $P^N_{\theta_n}(X_n,\cdot)$ to
be within $\epsilon$ to $\pi$ (by ergodicity). On the other hand, since the
adaptation is diminishing, by waiting long enough, we can find $n$ such that
the distribution of $X_{n+N}$ given $(X_n,\theta_n)$ is within $\epsilon$ to
$P^N_{\theta_n}(X_n,\cdot)$.  Combining these two arguments, we can then
conclude that the distribution of $X_{n+N}$ is within $2\epsilon$ to $\pi$.
This is essentially the argument of \cite{rosenthaletroberts05}. The difficulty
with this argument is that the distance between $P_{\theta_n}^N(x,\cdot)$ and
$\pi$ depends in general on $x$ and can rarely be bounded uniformly in $x$. We
solve this problem here by introducing some level set $\Cset$ of $V$ and by
using two basic facts: \textit{(i)} under A\ref{A-VCset}(\ref{Anew}), the
process cannot wait too long before coming back in $\Cset$; \textit{(ii)} under
A\ref{A-VCset}(\ref{A4rev}-\ref{A3rev}), a bound on the distance between
$P_{\theta_n}^N(x,\cdot)$ and $\pi$ uniformly in $x$, for $x \in \Cset$, is
possible.

The proof of Theorem \ref{theo:SLLNUnboundedUnifCase} is based on a resolvent
kernel approach that we adapted from \cite{merlevedeetal06} (see also
\cite{mw00}), combined with martingale theory. Another possible route to the
SLLN is the Poisson's equation technique which has been used to study adaptive
MCMC in \cite{andrieuetal06}.  Under A\ref{A2} and A\ref{A5}, a solution
$g_\theta$ to the Poisson's equation with transition kernel $P_\theta$ exists
for any $f\in\mathcal{L}_{V^\beta}$, $0\leq \beta\leq 1-\alpha$ and
$g_\theta\in\mathcal{L}_{V^{\beta+\alpha}}$. But in order to use
$\{g_\theta,\;\theta\in\Theta\}$ to obtain a SLLN for $f$, we typically need to
control $|g_\theta-g_{\theta'}|$ which overall can be expensive.  Here we avoid
these pitfalls by introducing the resolvent $\hat g_a(x,\theta)$ of the process
$\{X_n\}$, defined by
\[\hat g_a^{(l)}(x,\theta) \eqdef \sum_{j\geq 0}(1-a)^{j+1}\PE_{x,\theta}^{(l)}\left[f(X_j)\right] \eqsp, \;\;x\in\Xset,\theta\in\Theta,a\in(0,1), l \geq 0 \eqsp.
\]

\section{Examples}
\label{sec:Example}


\subsection{A toy example} We first consider an example discussed in
\cite{atchadeetrosenthal03} (see also \cite{rosenthaletroberts05}). Let $\pi$
be a target density on the integers $\{1, \cdots, K \}$, $K \geq 4$. Let
$\{P_\theta, \theta \in \{1, \cdots, M\} \}$ be a family of Random Walk
Metropolis algorithm with proposal distribution $q_\theta$, the uniform
distribution on $\{x-\theta, \cdots, x-1, x+1, \cdots, x+\theta \}$.

Consider the sequence $\{(X_n,\theta_{n}), n\geq 0 \}$ defined as follows:
given $X_n,\theta_n$,
\begin{itemize}
\item the conditional distribution of $X_{n+1}$ is $P_{\theta_n}(X_n, \cdot)$.
\item if $X_{n+1} = X_n$, set $\theta_{n+1} = \max(1, \theta_n -1)$ with
  probability $p_{n+1}$ and $\theta_{n+1} = \theta_n$ otherwise; if $X_{n+1}
  \neq X_n$, set $\theta_{n+1} = \min(M, \theta_n +1)$ with probability
  $p_{n+1}$ and $\theta_{n+1} = \theta_n$ otherwise.
\end{itemize}
This algorithm defines a non-homogeneous Markov chain - still denoted
$\{(X_n,\theta_{n}), n\geq 0 \}$ - on a canonical probability space endowed
with a probability $\PP$. The transitions of this Markov process are given by
the family of transition kernels $\{\bar P(n; (x,\theta), (dx', d\theta'),
n\geq 0 \}$ where
\begin{multline*}
  \bar P(n; (x,\theta), (dx', d\theta') = P_\theta(x,dx') \; \left( \un_{x=x'} \left\{ p_{n+1} \ \delta_{1 \vee (\theta-1)}(d\theta')   + (1-p_{n+1}) \ \delta_{\theta}(d\theta')  \right\}   \right. \\
  \left. + \un_{x\neq x'} \left\{ p_{n+1} \ \delta_{M \wedge
        (\theta+1)}(d\theta') + (1-p_{n+1}) \ \delta_{\theta}(d\theta')
    \right\} \right) \eqsp.
\end{multline*}

In this example, each kernel $P_\theta$ is uniformly ergodic~: $P_\theta$ is
$\phi$-irreducible, aperiodic, possesses an invariant probability measure $\pi$
and
\[
\lim_n \sup_{x \in \Xset} \|P_\theta^n(x,\cdot) - \pi(\cdot) \|_{\tv} = 0 \eqsp.
\]
Since $\Theta$ is finite, this implies that A\ref{A-VCset}(\ref{A4rev}) (resp.
A\ref{A6}) hold with $V=1$ (resp. $\Dset = \Xset$ and $\beta =0$). Furthermore,
$\PE_{x,\theta}^{(l)}\left[D(\theta_n, \theta_{n+1})\right] \leq 2 p_{n+1}$ so
that B\ref{B1} (resp. B\ref{B2}) hold with any probability measures $\xi_1,
\xi_2$ (resp.  with $\Dset = \Xset$) provided $p_n \to 0$. By
Lemma~\ref{lemma:MarginalUnifCaseBounded} combined with
Theorem~\ref{theo:MarginalUnifCase}, and by
Proposition~\ref{prop:SLLNUnboundedUnifCaseBounded}, we have
\begin{prop}
  Assume $\lim_n p_n =0$. For any probability distributions $\xi_1, \xi_2$ on
  $\Xset, \Theta$,
  \begin{enumerate}[(i)]
  \item $\sup_{\{f, |f|_1 \leq 1 \}} |\PE_{\xi_1,\xi_2}[f(X_n)] - \pi(f)| \to
    0$
  \item For any bounded function $f$
\[
n^{-1} \sum_{k=1}^n f(X_k) \to \pi(f) \eqsp, \qquad \qquad \PP_{\xi_1,\xi_2}-\as
\]
  \end{enumerate}
\end{prop}

\subsection{The adaptive Random Walk Metropolis of \cite{haarioetal00}}
\label{sec:ex2}
We illustrate our results with the adaptive Random Walk Metropolis of
\cite{haarioetal00}. The Random Walk Metropolis (RWM) algorithm is a popular
MCMC algorithm~\cite{hastings:1970,metropolis:1953}.  Let a target density
$\pi$, absolutely continuous w.r.t. the Lebesgue measure $\mu_{Leb}$ with
density still denoted by $\pi$.  Choose a proposal distribution with density
w.r.t. $\mu_{Leb}$ denoted $q$, and assume that $q$ is a positive symmetric
density on $\rset^p$. The algorithm generates a Markov chain $\{X_n, n\geq 0\}$
with invariant distribution $\pi$ as follows.  Given $X_n=x$, a new value
$Y=x+Z$ is proposed where $Z$ is generated from $q(\cdot)$.  Then we either
'accept' $Y$ and set $X_{n+1}=Y$ with probability
$\alpha(x,Y)\eqdef\min\left(1,\pi(Y)/\pi(x)\right)$ or we 'reject' $Y$ and set
$X_{n+1}=x$.

For definiteness, we will assume that $q$ is a zero-mean multivariate Gaussian
distribution (this assumption can be replaced by regularity conditions and
moment conditions on the proposal distribution). Given a proposal distribution
with finite second moments, the convergence rate of the RWM kernel depends
mainly on the tail behavior of the target distribution $\pi$. If $\pi$ is
super-exponential in the tails with regular contours, then the RWM kernel is
typically geometrically ergodic (\cite{jarnerethansen98}).  Otherwise, it is
typically sub-geometric
(\cite{gersendeetmoulines00,gersendeetmoulines03,doucetal04}).

Define
\[
\mu_\star\eqdef\int_\Xset x \; \pi(x) \; \mu_{Leb}(dx) \eqsp, \qquad
\Sigma_\star\eqdef \int_\Xset xx^T \; \pi(x)\mu_{Leb}(dx) -\mu_\star \;
\mu_\star^{T} \eqsp,
\]
resp. the expectation and the covariance matrix of $\pi$ ($\cdot^T$ denotes the
transpose operation).  Theoretical results suggest setting the
variance-covariance matrix $\Sigma$ of the proposal distribution
$\Sigma=c_\star\Sigma_\star$ where $c_\star$ is set so as to reach the optimal
acceptance rate $\bar\alpha$ in stationarity (typically $\bar\alpha$ is set to
values around $0.3-0.4$). See e.g.  \cite{robertsetrosenthal01} for more
details. \cite{haarioetal00} have proposed an adaptive algorithm to learn
$\Sigma_*$ adaptively during the simulation. This algorithm has been studied in
detail in \cite{andrieuetal06} under the assumption that $\pi$ is
super-exponential in the tails. An adaptive algorithm to find the optimal value
$c_\star$ has been proposed in \cite{atchadeetrosenthal03} (see also
\cite{atchade05}) and studied under the assumption that $\pi$ is
super-exponential in the tails. We extend these results to cases where $\pi$ is
sub-exponential in the tails.

Let $\Theta_+$ be a convex compact of the cone of $p\times p$ symmetric
positive definite matrices endowed with the Shur norm $|\cdot|_\s$,
$|A|_\s\eqdef \sqrt{\mathrm{Tr}(A^T \, A)}$.  For example, for $\mathsf{a}, M >
0$, $\Theta_+ = \{ \text{$A+\mathsf{a} \, \mathrm{Id}$: $A$ is symmetric
  positive semidefinite and } |A|_s \leq M \}$.  Next, for
$-\infty<\kappa_l<\kappa_u<\infty$ and $\Theta_\mu$ a compact subset of
$\Xset$, we introduce the space $\Theta \eqdef \Theta_\mu \times \Theta_+\times
[\kappa_l,\kappa_u]$. For $\theta =(\mu,\Sigma,c)\in \Theta$, denote by
$P_\theta$ the transition kernel of the RWM algorithm with proposal
$q_{\theta}$ where $q_\theta$ stands for the multivariate Gaussian distribution
with variance-covariance matrix $e^c \Sigma$.

Consider the adaptive RWM defined as follows

\begin{algo}\label{arwm1}
\begin{description}
\item [Initialization] Let $\bar\alpha$ be the target acceptance probability.
  Choose $X_0\in\Xset$, $(\mu_0,\Sigma_0,c_0)\in\Theta$.
\item [Iteration]
Given $(X_n,\mu_n,\Sigma_n,c_n)$:
\begin{description}
\item [1] Generate $Z_{n+1}\sim q_{\theta_n} d\mu_{Leb}$ and set $Y_{n+1} = X_n
  +Z_{n+1}$. With probability $\alpha(X_n,Y_{n+1})$ set $X_{n+1}=Y_{n+1}$ and
  with probability $1-\alpha(X_n,Y_{n+1})$, set $X_{n+1}=X_n$.
\item [2] Set
  \begin{align}
    \mu & = \mu_n+(n+1)^{-1}\left(X_{n+1}-\mu_n\right) \eqsp, \label{ex2:defiMu} \\
    \Sigma & = \Sigma_n+(n+1)^{-1}\left[\left(X_{n+1}-\mu_n\right)\left(X_{n+1}-\mu_n\right)^T-\Sigma_n\right] \eqsp, \label{ex2:defiSigma} \\
    c & = c_n+\frac{1}{n+1}\left(\alpha(X_n,Y_{n+1})-\bar\alpha\right) \eqsp.
    \label{ex2:defic}
  \end{align}
\item [3] If $(\mu, \Sigma,c)\in\Theta$, set $\mu_{n+1} = \mu$,
  $\Sigma_{n+1}=\Sigma$ and $c_{n+1}=c$.  Otherwise, set $\mu_{n+1} = \mu_n$,
  $\Sigma_{n+1}=\Sigma_n$ and $c_{n+1}=c_n$.
\end{description}
\end{description}
\end{algo}

This is an algorithmic description of a random process $\{(X_n, \theta_n),
n\geq 0\}$ which is a non-homogeneous Markov chain with successive transitions
kernels $\{\bar P(n; (x,\theta), (dx',d \theta')), n\geq 0 \}$ given by
\begin{multline*}
  \bar P(n; (x,\theta), (dx',d \theta')) = \int q_\theta(z)  \ \left\{ \alpha(x,x+z) \delta_{x+z}(dx')  + (1-\alpha(x,x+z)) \delta_x(dx') \right\} \cdots \\
  \left(\un_{\{\phi(\theta,x+z,x') \in
      \Theta\}}\delta_{\phi(\theta,x+z,x')}(d\theta') + \un_{\{\phi(\theta,x+z,x')
      \notin \Theta\}}\delta_{\theta}(d\theta') \right) \ d\mu_{Leb}(dz)
\end{multline*}
where $\phi$ is the function defined from the rhs expressions of
(\ref{ex2:defiMu}) to (\ref{ex2:defic}). Integrating over $\theta'$, we see
that for any $A \in \Xsigma$,
\[
\int_{A \times \Theta} \bar P(n;(x,\theta),(dx',d\theta')) = P_\theta(x,A)
\eqsp.
\]
\begin{lemma}
\label{lem:example:smallset}
  Assume that $\pi$ is bounded from below and from above on compact sets. Then
  any compact subset $\Cset$ of $\Xset$ with $\mu_{Leb}(\Cset)>0$ satisfies
  A\ref{A2}.
\end{lemma}
\begin{proof}
  See \cite[Theorem 2.2]{robertsettweedie96}.
\end{proof}

Following (\cite{gersendeetmoulines00}), we assume that $\pi$ is
sub-exponential in the tails: \debutD
\item \label{D1} $\pi$ is positive and continuous on $\rset^p$, and twice
  continuously differentiable in the tails.
\item \label{D2} there exist $m\in (0,1)$, positive constants $d_i<D_i$, $i=0,1,2$ and
  $r,R>0$ such that  for $|x|\geq R$:
\begin{enumerate}[(i)]
\item \label{D2z} $\pscal{\frac{\nabla \pi(x)}{|\nabla \pi(x)|}}{\frac{x}{|x|}}
  \leq -r $.
\item  \label{D2a}  $d_0|x|^m\leq -\log\pi(x)\leq D_0|x|^m$,
\item \label{D2b}   $d_1|x|^{m-1}\leq |\nabla\log\pi(x)|\leq D_1|x|^{m-1}$,
\item \label{D2c} $d_2|x|^{m-2}\leq |\nabla^2\log\pi(x)|\leq D_2|x|^{m-2}$.
\end{enumerate}
\finD

Examples of target density that satisfies D\ref{D1}-D\ref{D2} are the Weibull
distributions on $\rset$ with density $\pi(x) \propto |x|^{m-1} \exp(-\beta
|x|^m)$ (for large $|x|$), $\beta>0$, $m \in (0,1)$. Multidimensional examples
are provided in \cite{gersendeetmoulines00}.

\subsubsection{Law of large numbers for exponential functions}
In this subsection, we assume that \debutD
\item \label{D3} there exist $s_\star>0$, $ 0<\upsilon<1-m$ and $0<\eta<1$ such
  that as $|x| \to+\infty$,
 \[
 \sup_{\theta \in \Theta} \ \int_{\{z, |z| \geq \eta |x|^\upsilon \}} \left(1
    \vee \frac{\pi(x)}{\pi(x+z)} \right)^{s_\star} \; \; q_\theta(z) \
 \mu_{Leb}(dz) =o\left( |x|^{2(m-1)} \right) \eqsp.
\]
\finD A sufficient condition for D\ref{D3} is that $\pi(x+z) \geq \pi(x)
\pi(z)$ for any $x$ large enough and $|z| \geq \eta |x|^\upsilon$ (which holds
true for Weibull distributions with $0<m<1$).  Indeed, we then have
\begin{multline*}
  \int_{\{z, |z| \geq \eta |x|^\upsilon \}} \left(1 \vee
    \frac{\pi(x)}{\pi(x+z)} \right)^{s_\star} \; q_\theta(z) \mu_{Leb}(dz)  \\
  \leq C\; \exp(-\lambda_\star \eta^2 |x|^{2 \upsilon}) \sup_{\theta \in
    \Theta} \ \int \exp(s_\star D_0 |z|^m) \; \exp(\lambda_\star |z|^2) \
  q_\theta(z) \mu_{Leb}(dz)
\end{multline*}
for some constant $C< +\infty$, and $\lambda_\star >0$ such that the rhs is
finite.

\begin{lemma}\label{driftRWM}
  Assume D\ref{D1}-\ref{D3}.  For $0<s \leq s_\star$, define $V_s(x)\eqdef 1 +
  \pi^{-s}(x)$. There exist $0< s \leq s_\star$ and for any $\alpha \in (0,1)$,
  there exist positive constants $b,c$ and a compact set $\Cset$ such that
\begin{equation*}
  \sup_{\theta \in \Theta} P_\theta V_s(x)\leq
  V_s(x)-c V^{1-\alpha}_s(x)+b\un_\Cset(x).
\end{equation*}
Hence A\ref{Adrift}-\ref{A6} hold.
\end{lemma}
\begin{lemma} \label{ex:lem:HypB}
  Assume D\ref{D1}-\ref{D3}.  B\ref{B2} holds and B\ref{B1} holds for any
  probability measures $\xi_1$,$\xi_2$ such that $\int |\ln \pi|^{2/m} d \xi_1
  < +\infty$.
\end{lemma}
The proof of Lemmas~\ref{driftRWM} and \ref{ex:lem:HypB} are in Appendix C.

\begin{prop}
\label{prop:ex2:CasRapide}
  Assume D\ref{D1}-\ref{D3}. Consider the sequence $\{X_n, n \geq 0 \}$ given by the
  algorithm \ref{arwm1}.
  \begin{enumerate}[(i)]
  \item For any probability measures $\xi_1,\xi_2$ such that $\int |\ln
    \pi|^{2/m} d \xi_1 < +\infty$,
\[
\sup_{\{f, |f|_1 \leq 1 \}} |\PE_{\xi_1,\xi_2}[f(X_n)] - \pi(f)| \to 0 \eqsp.
\]
\item \label{item2} There exists $0 < s \leq s_\star$ such that for any
  probability measures $\xi_1,\xi_2$ such that $\int |\pi|^{-s} d \xi_1 <
  +\infty$, and any function $ f \in \L_{1+\pi^{-r}}$, $0 \leq r<s$,
\[
n^{-1} \sum_{k=1}^n f(X_k ) \to \pi(f) \eqsp, \qquad \PP_{\xi_1,\xi_2}-\as
\]
  \end{enumerate}
\end{prop}
The drift function $V_s$ exhibited in Lemma 3.3. is designed for limit theorems
relative to functions $f$ increasing as $\exp(\beta |x|^m)$. This implies a
condition on the initial distribution $\xi_1$ which has to possess
sub-exponential moments (see
Proposition~\ref{prop:ex2:CasRapide}(\ref{item2})), which always holds with
$\xi_1 = \delta_x$, $ x \in \Xset$.

\subsubsection{Law of large numbers for polynomially increasing functions}
Proposition~\ref{prop:ex2:CasRapide} also addresses the case when $f$ is of the
form $1+|x|^r$, $r>0$. Nevertheless, the conditions on $\xi_1$ and the
assumptions D\ref{D3} can be weakened in that case.

We have to find a drift function $V$ such that $V^{1-\alpha}(x) \sim
1+|x|^{r+\iota}$ for some $\alpha \in (0,1)$, $\iota>0$.  Under D\ref{D3}, this
can be obtained from the proof of Lemma 3.3. and this yields $V(x) \sim 1 +
|x|^{r+\iota+2-m}$ (apply the Jensen's inequality to the drift inequality
(\ref{eq:drift:sous-geom}) with the concave function $\phi(t) \sim [\ln
t]^{(r+\iota+2)/m-1}$; see \cite[Lemma 3.5]{jarneretroberts02} for similar
calculations). Hence, the condition on $\xi_1$ gets into
$\xi_1(|x|^{r+\iota+2-m})< +\infty$ for some $\iota>0$.

Drift inequalities with $V \sim (-\ln \pi)^{s}$ for some $s>2/m-1$, can also be
derived by direct computations: in that case, D\ref{D3} can be removed.
Details are omitted and left to the interested reader.

To conclude, observe that these discussions relative to polynomially increasing
functions can be extended to any function $f$ which is a concave transformation
of $\pi^{-s}$.

\section{Proofs of the results of Section~\ref{sec:ResultsUnif}}
\label{sec:Proofs}
For a set $\Cset \in \Xsigma$, define the hitting-time on $\Cset \times \Theta$
of $\{Z_n, n\geq 0 \}$ by $\sigma_\Cset \eqdef \inf\{n \geq 0, Z_n \in \Cset
\times \Theta \}$. If $\pi(|f|) < +\infty$, we set $\bar f \eqdef f - \pi(f)$.
\subsection{Preliminary results}
We gather some useful preliminary results in this section. Section
\ref{sec:OptCouplingUnif} gives an approximation of the marginal distribution
of the adaptive chain by the distribution of a related Markov chain. In Section
\ref{sec:ModMomentsUnif}, we develop various bounds for modulated moments of
the adaptive chain as consequences of the drift conditions. In Section
\ref{sec:ReturnTimesUnif} we bound the expected return times of the adaptive
chain to level sets of the drift function $V$. The culminating result of this
subsection is Theorem~\ref{theo:controleG} which gives an explicit bound on the
resolvent function $g^{(l)}_a(x,\theta)$.

\subsubsection{Optimal coupling}\label{sec:OptCouplingUnif}
\begin{lemma}
  \label{lem:coupling}
  For any integers $l \geq 0, N \geq 2$, any measurable bounded function $f$ on
  $\Xset^N$ and any $(x,\theta) \in \Xset \times \Theta$,
  \begin{multline*}
    \Delta \eqdef \left| \PE_{x, \theta}^{(l)}\left[ f(X_1, \cdots, X_N)
      \right] - \int_{\Xset^N} P_{\theta}(x, dx_1) \; \prod_{k=2}^N
      P_{\theta}(x_{k-1}, dx_k) f(x_1, \cdots, x_n)\right| \\ \leq |f|_1 \;
    \sum_{j=1}^{N-1} \sum_{i=1}^j \PE_{x,\theta}^{(l)} \left[D(\theta_i,
      \theta_{i-1}) \right] \eqsp.
  \end{multline*}
\end{lemma}
\begin{proof}
  We can assume w.l.g. that $|f|_{1} \leq 1$.  Set $z_k = (x_k,t_k)$.  With the
  convention that $\prod_{k=a}^b a_k=1$ for $a>b$ and upon noting that $
  \int_\Xset P_{\theta}(x, dx') h(x') = \int_{\Xset \times \Theta} \bar
  P_{l}(0; (x,\theta), (dx', d \theta'))h(x')$ for any bounded measurable function $h:\;\Xset\to\rset$,
\begin{multline*}
  \Delta = \left| \int_{(\Xset \times \Theta)^N} \sum_{j=1}^{N-1} \bar P_{l}(0;
    (x,\theta), dz_1) \;
    \prod_{k=2}^{j} \bar P_{l}(k-1; z_{k-1}, d z_{k})  \cdots  \right. \\
  \left.  \left\{ \bar P_{l}(j; z_j, dz_{j+1}) - \bar P_{l}(0; (x_j,\theta),
      dz_{j+1}) \right\}
    \prod_{k=j+2}^N   \bar P_{l}(0; (x_{k-1},\theta), dz_{k})  f(x_1, \cdots, x_N) \right| \\
  \leq \sum_{j=1}^{N-1} \int_{\Xset^j} \bar P_{l}(0; (x,\theta), dz_1) \;
  \prod_{k=2}^{j} \bar P_{l}(k-1; z_{k-1}, d z_{k}) \sup_{x\in \Xset} \|
  P_{t_j}( x, \cdot) - P_{\theta}(x, \cdot) \|_{\tv}
\end{multline*}
where we used that
\[
\int_{(\Xset \times \Theta)^{N-j-1}} \prod_{k=j+2}^N \bar P_{l}(0;
(x_{k-1},\theta), dz_{k}) f(x_1, \cdots, x_N)
\]
is bounded by a function $\Xi(x_{1}, \cdots, x_{j+1})$ that does not depend
upon $t_k, k\leq N$ and for any bounded function $\Xi$ on $\Xset^{j+1}$
\begin{multline*}
  \int_{\Xset \times \Theta} \left\{ \bar P_{l}(j; z_j, dz_{j+1}) - \bar
    P_{l}(0;
    (x_j,\theta), dz_{j+1}) \right\} \Xi(x_1, \cdots, x_{j+1}) \\
  = \int_{\Xset} \left\{ P_{t_j}( x_j, dx_{j+1}) - P_{\theta}( x_j, dx_{j+1})
  \right\} \Xi(x_1, \cdots, x_{j+1}) \leq \sup_{x\in \Xset} \| P_{t_j}( x,
  \cdot) - P_{\theta}(x, \cdot) \|_{\tv} \ |\Xi|_1\eqsp.
\end{multline*}
Hence
\begin{multline*}
  \Delta \leq \sum_{j=1}^{N-1} \PE_{x,\theta}^{(l)} \left[ \sup_{x\in \Xset} \|
    P_{\theta_j}(x, \cdot) - P_{\theta_0}(x, \cdot)
    \|_{\tv} \right]   \\
  \leq \sum_{j=1}^{N-1} \PE_{x,\theta}^{(l)} \left[ \sum_{i=1}^j \sup_{x\in
      \Xset} \| P_{\theta_i}(x, \cdot) - P_{\theta_{i-1}}(x, \cdot) \|_{\tv}
  \right] = \sum_{j=1}^{N-1} \sum_{i=1}^j \PE_{x,\theta}^{(l)}
  \left[D(\theta_i, \theta_{i-1}) \right] \eqsp.
\end{multline*}
\end{proof}

\begin{lemma}
\label{lem:couplingoptimal}
Let $\mu, \nu$ be two probability distributions. There exist a probability
space $(\Omega, \F, \PP)$ and random variables $X,Y$ on $(\Omega, \F)$ such
that $X \sim \mu$, $Y \sim \nu$ and $\PP(X = Y) = 1 - \| \mu - \nu \|_\tv$.
\end{lemma}
The proof can be found e.g. in \cite[Proposition 3]{roberts:rosenthal:2004}.
As a consequence of Lemmas~\ref{lem:coupling} and \ref{lem:couplingoptimal}, we
have
\begin{prop}
  \label{prop:ContructionCouplingOpt}
  Let $l \geq 0, N \geq 2$ and set $z = (x,\theta)$. There exists a process
  $\{(X_k, \tilde X_k), 0 \leq k \leq N\}$ defined on a probability space
  endowed with the probability $\bPP_{z, z}^{(l)}$ such that
\[
\bPP_{z, z}^{(l)} \left( X_k = \tilde X_k, 0 \leq k \leq N \right) \geq 1 -
\sum_{j=1}^{N-1} \sum_{i=1}^j\PE_{z}^{(l)} \left[ D(\theta_i,\theta_{i-1})
\right] \eqsp,
\]
$(X_0, \cdots, X_{N})$ has the $X$-marginal distribution of $\PP^{(l)}_{z}$
restricted to the time-interval $\{0, \cdots, N\}$, and $(\tilde X_0, \cdots,
\tilde X_{N})$ has the same distribution as a homogeneous Markov chain with
transition kernel $P_{\theta}$ and initial distribution $\delta_x$.
\end{prop}

\subsubsection{Modulated moments for the adaptive chain}\label{sec:ModMomentsUnif}

Let $V: \Xset \to [1, +\infty)$ be a measurable function and assume that there
exist $\Cset \in \Xsigma$, positive constants $b,c$ and $0 < \alpha \leq 1$
such that for any $\theta \in \Theta$,
\begin{equation}
  \label{eq:A2-A5}
  P_\theta V \leq V - c V^{1-\alpha} +b \un_\Cset \eqsp.
\end{equation}

\begin{lemma}
\label{lem:JarnerRoberts}
Assume (\ref{eq:A2-A5}). There exists $\bar b$ such that for any $0 \leq \beta
\leq 1$, $\theta \in \Theta$: $P_\theta V^\beta \leq V^\beta - \beta c
V^{\beta-\alpha} + \bar b \un_\Cset$.
\end{lemma}
\begin{proof}
  See \cite[Lemma 3.5]{jarneretroberts02}.
\end{proof}
\begin{prop}
  Assume (\ref{eq:A2-A5}). For any $l\geq 0$, $(x,\theta) \in \Xset \times
  \Theta$, and any stopping-time $\tau$,
\[
c \ \PE_{x,\theta}^{(l)} \left[ \sum_{k=0}^{\tau-1} \left(k \alpha c + 1
  \right)^{\alpha^{-1}-1} \right] \leq V(x) + b \ \PE_{x,\theta}^{(l)} \left[
  \sum_{k=0}^{\tau-1} \left((k+1) \alpha c + 1 \right)^{\alpha^{-1}-1}
  \un_\Cset(X_k)\right] \eqsp.
\]
\end{prop}
\begin{proof}
  The proof can be adapted from \cite[Proposition 2.1]{doucetal04} and
  \cite[Proposition 11.3.2]{meynettweedie93}and is omitted.
\end{proof}

\begin{prop}
\label{prop:ComparaisonGal} Assume (\ref{eq:A2-A5}).
  \begin{enumerate}[(i)]
  \item \label{prop:CpG1} There exists $\bar b$ such that for any $j \geq 0$,
    $0 \leq \beta \leq 1$, $l\geq 0$ and $(x,\theta) \in \Xset \times \Theta$
\[
\PE_{x,\theta}^{(l)} \left[ V^\beta(X_j)\right] \leq V^\beta(x) + \bar b
j^\beta \eqsp.
\]
\item \label{prop:CpG2} Let $0 \leq \beta \leq 1$ and $0 \leq a \leq 1$. For
  any stopping-time $\tau$,
    \begin{multline*}
      \PE_{x,\theta}^{(l)} \left[(1-a)^\tau V^\beta(X_{\tau}) \un_{\tau <
          +\infty}\right] + \PE_{x,\theta}^{(l)} \left[ \sum_{j=0}^{\tau-1}
        (1-a)^{j} \; \{ a \; V^\beta(X_{j}) + \beta c (1-a)
        V^{\beta-\alpha}(X_{j}) \} \right] \\
      \leq V^\beta(x) + \bar b (1-a) \PE_{x,\theta}^{(l)} \left[
        \sum_{j=0}^{\tau-1} (1-a)^{j} \; \un_\Cset(X_{j}) \right] \eqsp..
    \end{multline*}
  \item \label{prop:CpG3} Let $0 \leq \beta \leq 1-\alpha$ and $0<a<1$. For any
    stopping-time $\tau$ and any $q \in [1, +\infty]$,
    \begin{multline*}
      \PE_{x,\theta}^{(l)}  \left[ \sum_{j=0}^{\tau-1} (1-a)^{j} V^\beta(X_{j}) \right] \\
      \leq a^{1/q-1} (1-a)^{-1/q} \; V^{\beta+\alpha/q}(x) \; \left(1 + \bar b
        \; \PE_{x,\theta}^{(l)}\left[\sum_{j=0}^{\tau-1} (1-a)^j \un_\Cset(X_j)
        \right] \right) \left( \alpha c \right)^{-1/q} \eqsp,
    \end{multline*}
    (with the convention that $1/q = 0$ when $q = +\infty$).
  \end{enumerate}
\end{prop}

\begin{proof} The proof is done in the case $l=0$. The general case is similar and omitted.
  (\ref{prop:CpG1}) is a trivial consequence of Lemma~\ref{lem:JarnerRoberts}.
  (\ref{prop:CpG2}) Let $\beta \leq 1$.  Set $\tau_N = \tau \wedge N$ and $Y_n
  = (1-a)^n V^\beta(X_n)$. Then
  \begin{multline*}
    Y_{\tau_N} = Y_0 + \sum_{j=1}^{\tau_N} \left(Y_j - Y_{j-1} \right) = Y_0 +
    \sum_{j=1}^{\tau_N} (1-a)^{j-1} \; \left((1-a) V^\beta(X_j) -
      V^\beta(X_{j-1}) \right) \\
    = Y_0 + \sum_{j=1}^{\tau_N} (1-a)^{j} \; \left(V^\beta(X_j) -
      V^\beta(X_{j-1}) \right) - a \sum_{j=1}^{\tau_N} (1-a)^{j-1} \;
    V^\beta(X_{j-1}) \eqsp.
  \end{multline*}
Hence,
\begin{multline*}
  \PE_{x,\theta} \left[Y_{\tau_N} \right] + a \; \PE_{x,\theta} \left[
    \sum_{j=0}^{\tau_N-1} (1-a)^{j} \; V^\beta(X_{j}) \right] \\
  = V^\beta(x) + \sum_{j\geq 1} (1-a)^{j} \; \PE_{x,\theta} \left[
    \left(V^\beta(X_j) - V^\beta(X_{j-1})
    \right) \un_{j \leq \tau_N} \right] \\
  \leq V^\beta(x) + \sum_{j\geq 1} (1-a)^{j} \; \PE_{x,\theta} \left[\left( -
      \beta c \; V^{\beta-\alpha}(X_{j-1}) + \bar b \un_\Cset(X_{j-1}) \right)
    \un_{j \leq \tau_N} \right],
\end{multline*}
where we used Lemma~\ref{lem:JarnerRoberts} in the last inequality.  This
implies
\begin{multline*}
  \PE_{x,\theta} \left[Y_{\tau_N} \right] + a \; \PE_{x,\theta} \left[
    \sum_{j=0}^{\tau_N-1} (1-a)^{j} \; V^\beta(X_{j}) \right] + (1-a) \beta c
  \; \PE_{x,\theta} \left[ \sum_{j=0}^{\tau_N-1} (1-a)^{j} \;
    V^{\beta-\alpha}(X_{j}) \right] \\
  \leq V^\beta(x) + \bar b (1-a) \PE_{x,\theta} \left[ \sum_{j=0}^{\tau_N-1}
    (1-a)^{j} \; \un_\Cset(X_{j}) \right].
\end{multline*}
The results follows when $N \to +\infty$. \\
(\ref{prop:CpG3}) The previous case provides two upper bounds, namely for $0 <
\beta \leq 1-\alpha$,
\[
a \; \PE_{x,\theta} \left[ \sum_{j=0}^{\tau-1} (1-a)^{j} V^\beta(X_{j}) \right]
\leq V^\beta(x) + \bar b \; (1-a) \PE_{x,\theta} \left[ \sum_{j=0}^{\tau-1}
  (1-a)^{j} \; \un_\Cset(X_{j}) \right],
\]
and
\[
(1-a) \; \left( (\beta +\alpha) c \right) \ \PE_{x,\theta} \left[
  \sum_{j=0}^{\tau-1} (1-a)^{j} V^{\beta}(X_{j}) \right] \leq
V^{\beta+\alpha}(x) + \bar b \PE_{x,\theta} \left[ \sum_{j=0}^{\tau-1}
  (1-a)^{j} \; \un_\Cset(X_{j}) \right].
\]
We then use the property $\left[c \leq c_1 \wedge c_2 \right]\Longrightarrow c
\leq c_1^{1/q} c_2^{1-1/q}$ for any $ q \in [1, +\infty]$.
\end{proof}

\begin{prop}
  \label{prop:ComparaisonGal2} Assume (\ref{eq:A2-A5}). Let $\{r_n, n\geq 0\}$ be a non-increasing positive sequence. There exists $\bar b$ such that for any $l\geq0$, $(x,\theta) \in \Xset \times \Theta$, $ 0 \leq \beta \leq 1$  and $n \geq 0$,
\[
\beta c \ \PE_{x,\theta}^{(l)} \left[ \sum_{k \geq n} r_{k+1}
  V^{\beta-\alpha}(X_k) \right] \leq r_n \PE_{x,\theta}^{(l)} \left[
  V^\beta(X_n)\right] + \bar b \ \PE_{x,\theta}^{(l)} \left[ \sum_{k \geq n}
  r_{k+1} \un_\Cset(X_k) \right] \eqsp.
\]
\end{prop}
The proof is on the same lines as the proof of
Proposition~\ref{prop:ComparaisonGal}(\ref{prop:CpG2}) and is omitted.

\subsubsection{Delayed successive visits to an accessible level set of $V$}\label{sec:ReturnTimesUnif}
\label{sec:DelayedSuccVisit}
Let $\Dset \in \Xsigma$ and two positive integers $n_\star, N$.  Define on
$(\Omega, \F, \PP_{x,\theta}^{(l)})$ the sequence of $\nset$-valued random
variables $\{\tau^n, n\geq 1 \}$ as
\[
\tau^0 \eqdef \tau_\Dset \eqsp, \qquad \tau^1 \eqdef \tau^0 + n_\star +
\tau_\Dset \circ \underline{\theta}^{\tau^0 + n_\star} \eqsp, \qquad \tau^{k+1}
\eqdef \tau^k + N + \tau_\Dset \circ \underline{\theta}^{\tau^k+N} \eqsp, \ \
k\geq 1 \eqsp.
\]

\begin{prop}
\label{prop:TimeFiniteAS}
Assume A\ref{A2} and there exist $V : \Xset \to [1,+\infty)$ and a constant $b
< +\infty$ such that for any $\theta \in \Theta$, $P_\theta V \leq V - 1 + b
\un_\Cset$. Let $\Dset \in \Xsigma$.  Let $n_\star, N$ be two non-negative
integers.  Then
\[
\varepsilon \ \nu(\Dset) \ \PE_{x,\theta}^{(l)} \left[
  \sum_{k=0}^{\tau_\Dset-1} \un_\Cset(X_k)\right] \leq 1 \eqsp,
\]
and if $\sup_\Dset V < +\infty$ and $\nu(\Dset)>0$, there exists a (finite)
constant $C$ depending upon $\varepsilon, \nu(\Dset), \sup_\Dset V, b, n_\star,
N$ such that for any $l \geq 0$, $(x,\theta) \in \Xset \times \Theta$ and
$k\geq 0$,
\[
\PE_{x,\theta}^{(l)} \left[ \tau^k \right] \leq k \ C + V(x)\eqsp.
\]
\end{prop}
\begin{proof}
  Since $V \geq 1$, Proposition~\ref{prop:ComparaisonGal}(\ref{prop:CpG2})
  applied with $a=0$, $\beta=\alpha=1$, $c=1$ and $\tau = \tau_\Dset$ implies
\[
\PE_{x,\theta}^{(l)}\left[ \tau_\Dset \right] \leq V(x) + \bar b \
\PE_{x,\theta}^{(l)} \left[ \sum_{k=0}^{\tau_\Dset-1} \un_\Cset(X_k)\right]
\eqsp.
\] By A\ref{A2}, we have $P_\theta(x,\Dset) \geq [\varepsilon \nu(\Dset)] \  \un_\Cset(x)$  for any $(x,\theta)$ so that
\[
\varepsilon \nu(\Dset) \ \PE_{x,\theta}^{(l)} \left[ \sum_{k=0}^{\tau_\Dset-1}
  \un_\Cset(X_k)\right] \leq \PE_{x,\theta}^{(l)} \left[
  \sum_{k=0}^{\tau_\Dset-1} P_{\theta_k}(X_k,\Dset)\right] =
\PE_{x,\theta}^{(l)} \left[ \sum_{k=0}^{\tau_\Dset-1} \un_{\Dset}(X_{k+1})
\right] \leq 1 \eqsp.
\]
Hence $\PE_{x,\theta}^{(l)}\left[ \tau_\Dset \right] \leq V(x) + \bar b[\varepsilon
\nu(\Dset)]^{-1}$. By the Markov property and
Proposition~\ref{prop:ComparaisonGal}(\ref{prop:CpG1})
\begin{multline*}
  \PE_{x,\theta}^{(l)}\left[ \tau^1 \right] \leq n_\star + V(x) + \bar b
  [\varepsilon \nu(\Dset)]^{-1} + \PE_{x,\theta}^{(l)}\left[ \PE_{Z_{n_\star +
        \tau_\Dset}}^{(n_\star+l+\tau_\Dset)}\left[ \sigma_\Dset
    \right] \right] \\
  \leq n_\star + 2 \; \bar b[\varepsilon \nu(\Dset)]^{-1} + V(x) + \sup_\Dset V
  + n_\star \bar b\eqsp.
\end{multline*}
The proof is by induction on $k$.  Assume that $\PE_{x,\theta}^{(l)}\left[
  \tau^k \right] \leq k C + V(x)$ with $C \geq2 \bar b[\varepsilon
\nu(\Dset)]^{-1}+ \sup_\Dset V + (N  \vee n_\star)(1+\bar b)$. Then using again the Markov
property and Proposition~\ref{prop:ComparaisonGal}(\ref{prop:CpG1}), and upon
noting that $\PP_{x,\theta}^{(l)}(Z_{\tau^k} \in \Dset) =1$,
\begin{multline*}
  \PE_{x,\theta}^{(l)}\left[ \tau^{k+1} \right] \leq N +
  \PE_{x,\theta}^{(l)}\left[ \tau^{k} \right]+ \PE_{x,\theta}^{(l)}\left[
    \PE_{Z_{\tau^k+N}}^{(\tau^k+N+l)}\left[ \tau_\Dset \right] \right] \\
  \leq N + \bar b[\varepsilon \nu(\Dset)]^{-1} + \PE_{x,\theta}^{(l)}\left[
    \tau^{k}
  \right]+\PE_{x,\theta}^{(l)}\left[ V(X_{\tau^k+N}) \right] \\
  \leq N + \bar b[\varepsilon \nu(\Dset)]^{-1} + \PE_{x,\theta}^{(l)}\left[
    \tau^{k} \right]+\PE_{x,\theta}^{(l)}\left[
    \PE_{Z_{\tau^k}}^{(\tau^k+l)}\left[
      V(X_{N}) \right]\right]  \\
  \leq N + \bar b [\varepsilon \nu(\Dset)]^{-1} + \PE_{x,\theta}^{(l)}\left[
    \tau^{k} \right]+ \left( \sup_\Dset V + N \bar b \right) \eqsp.
\end{multline*}
\end{proof}

\subsubsection{Generalized Poisson equation}
\label{sec:GeneralPoisson}
Assume (\ref{eq:A2-A5}).  Let $0 < a <1$, $l\geq 0$ and $0\leq \beta \leq
1-\alpha$.  For $f \in \L_{V^\beta}$ such that $\pi(|f|) < +\infty$, let us
define the function
\[
\hat g_{a}^{(l)}(x,\theta) \eqdef \sum_{j \geq 0} (1-a)^{j+1} \;
\PE_{x,\theta}^{(l)}[\bar f(X_j)] \eqsp.
\]

\begin{prop}
\label{prop:QuasiPoissonEq}
Assume (\ref{eq:A2-A5}).  Let $0 \leq \beta \leq 1-\alpha$ and $f \in
\mathcal{L}_{V^\beta}$.  For any $(x,\theta) \in \Xset \times \Theta$, $l \geq
0$ and $0<a<1$, $\hat g_a^{(l)}$ exists, and
\[
\bar f(x) = \frac{1}{1-a} \hat g_a^{(l)}(x,\theta) -\PE_{x,\theta}^{(l)} \left[
  \hat g_a^{(l+1)}\left(X_1, \theta_1 \right) \right] \eqsp.
\]
\end{prop}
\begin{proof}
  By Proposition~\ref{prop:ComparaisonGal}(\ref{prop:CpG1}), $\left|
    \PE_{x,\theta}^{(l)} \left[ \bar f(X_j) \right]\right| \leq |\bar
  f|_{V^\beta} \; \left( V^\beta(x) + \bar b j^\beta \right) $.  Hence, $ \hat
  g_a^{(l)}(x,\theta)$ exists for any $x,\theta,l$. Furthermore, $\hat
  g_a^{(l+1)}\left(X_1, \theta_1 \right)$ is $\PP_{x,\theta}^{(l)}$-integrable.
  By definition of $\hat g_a^{(l)}$ and by the Markov property,
\begin{multline*}
  \PE_{x,\theta}^{(l)} \left[ \hat g_a^{(l+1)}\left(X_1, \theta_1 \right)
  \right] = \sum_{j \geq 0} (1-a)^{j+1} \PE_{x,\theta}^{(l)} \left[ \bar
    f(X_{j+1}) \right] = (1-a)^{-1}\; \sum_{j
    \geq 1} (1-a)^{j+1} \PE_{x,\theta}^{(l)} \left[ \bar f(X_{j}) \right] \\
  = (1-a)^{-1}\; \left( \hat g_a^{(l)}(x,\theta) - (1-a) \bar f(x) \right).
\end{multline*}
\end{proof}

\begin{theo}
\label{theo:controleG}
Assume A\ref{A2}-\ref{A6} and B\ref{B2}. Let $0 \leq \beta <1-\alpha$. For any
$\epsilon>0$, there exists an integer $n \geq 2$ such that for any $0<a<1$, $f
\in \L_{V^\beta}$, $l\geq 0$, $(x,\theta) \in \Xset \times \Theta$ and $q
\in[1, +\infty]$,
\begin{multline*}
  \left( |\bar f|_{V^\beta} \right)^{-1} \; \left| \hat g_{a}^{(l)}(x,\theta)
  \right| \leq 4 \; \epsilon \; \left(1-(1-a)^n \right)^{-1} \; n  \\
  + \frac{V^{\beta+\alpha/q}(x)}{a^{1-1/q}(1-a)^{1/q}} (\alpha c)^{-1/q}\;
  \left( 1+ \bar b [\varepsilon \nu(\Dset)]^{-1} + 2 \; (1+\bar b n_\star)
    (1+\bar b) \ \sup_\Dset V^{\beta +\alpha/q} \right) \eqsp.
\end{multline*}
By convention, $1/q =0$ when $q = +\infty$.  In particular, $\lim_{a \to 0}
\left( |\bar f|_{V^\beta} \right)^{-1} \; \left| a\hat g_{a}^{(l)}(x,\theta)
\right| = 0$.
\end{theo}

\begin{rem}\label{remtheoG}
Before dwelling into the proof of the theorem, we first make two important remarks. Firstly, a simplified restatement of Theorem \ref{theo:controleG} is the following. There exists a finite constant $c_0$ such that
for any $0<a\leq 1/2$, $f \in \L_{V^\beta}$, $l\geq 0$, $(x,\theta) \in \Xset \times
\Theta$ and $q \in[1, +\infty]$,
\begin{equation}
\label{eq:MajoRem}
\left| \hat g_{a}^{(l)}(x,\theta)\right|\leq c_0 |\bar f|_{V^\beta} \ a^{-1} \left(1+a^{1/q}  V^{\beta+\alpha/q}(x)\right).\end{equation}
This follows by taking $\epsilon=1$, say, and upon noting that
$n\left(1-(1-a)^n\right)^{-1}\leq 2^{n-1}/a$.
The second point is that if we take $a_1,a_2\in (0,1)$ we can write
\[
\hat g_{a_1}^{(l)}(x,\theta)-\hat g_{a_2}^{(l)}(x,\theta)=\frac{a_2-a_1}{(1-a_1)(1-a_2)}\times\\
  \sum_{k\geq 0}(1-a_1)^{k+1}\PE_{x,\theta}^{(l)}\left[ \hat
    g_{a_2}^{(l+k)}(X_k, \theta_k)\right]  \eqsp.
\]
By (\ref{eq:MajoRem}) and Proposition \ref{prop:ComparaisonGal}
(\ref{prop:CpG3}),  it holds 
 \begin{equation}\label{bounddiffGa}
 \left|\hat g_{a_1}^{(l)}(x,\theta)-\hat g_{a_2}^{(l)}(x,\theta)\right|\leq c_1 \ |\bar f|_{V^\beta} \ |a_2-a_1|a_2^{-1}a_1^{-2+1/q}V^{\beta+\alpha/q}(x),\end{equation}
 for some finite constant $c_1$, for all $0<a_1,a_2\leq 1/2$, $f \in \L_{V^\beta}$, $l\geq 0$, $(x,\theta) \in \Xset \times
\Theta$ and $q \in[1, +\infty]$.
\end{rem}

\begin{proof}
  Let $\epsilon>0$.  Let us consider the sequence of stopping times $\{\tau^k,
  k \geq 0\}$ defined in Section~\ref{sec:DelayedSuccVisit} where $(\Dset, N,
  n_\star)$ are defined below.

\paragraph{\tt Choice of $\Dset, N,  n_\star$.}
Choose a level set $\Dset$ of $V$ large enough so that $\nu(\Dset)>0$.  Choose
$N$ such that
\begin{equation}
  \label{eq:Controle3}
  \frac{1}{N} \; \sum_{j=0}^{N-1} \sup_{\Dset \times \Theta} \; \| P_\theta^j(x,\cdot) -
  \pi(\cdot) \|_{V^\beta} \leq   \epsilon \eqsp,
\end{equation}
the existence of which is given by A\ref{A6}; and such that - since $\alpha +
\beta <1$, -
  \begin{equation}
    \label{eq:Controle4}
(\alpha c)^{-1} \  N^{-1} \left( \sup_\Dset V^{\beta +\alpha}   + \bar b N^{\beta +\alpha} + \bar b [\varepsilon \nu(\Dset)]^{-1}  \right)  \leq  \epsilon \eqsp.
  \end{equation}
  Set $\epsilon_N \eqdef N^{-2} \{ \epsilon \; \left( \sup_\Dset V^\beta + \bar
    b N^{-1} \sum_{j=1}^{N-1} j^\beta \right)^{-1} \}^{1/(1-\beta)}$ (which can
  be assumed to be strictly lower than $N^{-2}$ since $\beta>0$).  By
  B\ref{B2}, choose $n_\star$ such that for any $q \geq n_\star$, $l\geq 0$, $
  \sup_{\Dset \times \Theta} \PP_{x,\theta}^{(l)}(D(\theta_q,\theta_{q-1}) \geq
  \epsilon_N/2) \leq \epsilon_N/4$.

  By Proposition~\ref{prop:TimeFiniteAS}, $\PP_{x,\theta}^{(l)}(\tau^k <
  +\infty) =1$ for any $(x,\theta) \in \Xset \times \Theta$, $l \geq 0$, $k\geq
  0$.

\paragraph{\tt Optimal coupling.}
With these definitions, $ \sup_{i \geq 1} \sup_{k \geq 1} \PE_{x,\theta}^{(l)}
\left[ \PE_{Z_{\tau^k}}^{(\tau^k+l)} \left[ D(\theta_i,\theta_{i-1}) \right]
\right]\leq \epsilon_N$, upon noting that $\PP_{x,\theta}^{(l)}( n_\star \leq
\tau^k) =1$ and $D(\theta,\theta') \leq 2$. We apply
Proposition~\ref{prop:ContructionCouplingOpt} and set $\mathcal{E}_N \eqdef
\{X_k = \tilde X_k, 0 \leq k<N \}$. We have for any $l \geq 0$, $k \geq 1$,
$(x,\theta) \in \Xset \times \Theta$,
 \begin{equation}
    \label{eq:CouplingProbability2}
  \PE_{x,\theta}^{(l)}\left[ \bPP_{Z_{\tau^k}, Z_{\tau^k}}^{(\tau^k+l)} \left( \mathcal{E}^c_N \right)  \right] \leq
\sum_{j=1}^{N-1}  \sum_{i=1}^j \PE_{x,\theta}^{(l)}\left[ \PE_{Z_{\tau^k}}^{(\tau^k+l)} \left[ D(\theta_i,\theta_{i-1}) \right] \right]  \leq N^2 \epsilon_N  < 1 \eqsp.
  \end{equation}
  Observe that $\Dset,N$ and $n_\star$ do not depend upon $a,l,x,\theta$ and
  $f$.

\paragraph{\tt  Proof of Theorem~\ref{theo:controleG}.}
Assume that for any $0<a<1$, $l\geq 0$, $(x,\theta) \in \Xset \times \Theta$
and $k \geq 2$,
\begin{equation}
  \label{eq:ResultPropAdaptive2}
\left| \PE_{x,\theta}^{(l)} \left[ \sum_{j=0}^{N-1} (1-a)^{\tau^k+ j+1} \; \bar f
    \left( X_{\tau^{k}+j} \right)\right] \right|\leq |\bar f|_{V^\beta} \; 3 N
\epsilon \; (1-a)^{n_\star+(k-1)N} \eqsp.
\end{equation}
We have
\[
\hat g_{a}^{(l)}(x,\theta) = \sum_{j \geq 0} (1-a)^{j+1} \left\{
  \PE_{x,\theta}^{(l)}\left[\bar f(X_j) \un_{j < \tau^1}\right] + \sum_{k \geq
    1} \PE_{x,\theta}^{(l)}\left[\bar f(X_j) \un_{\tau^k \leq j <
      \tau^{k+1}}\right] \right\} \eqsp.
\]
On one hand, by Proposition~\ref{prop:ComparaisonGal}(\ref{prop:CpG3}) applied
with $\tau = \tau_\Dset$ and Proposition~\ref{prop:TimeFiniteAS},
\begin{multline*}
  \left| \sum_{j \geq 0} (1-a)^{j+1} \PE_{x,\theta}^{(l)}\left[\bar f(X_j)
      \un_{j < \tau^0}\right] \right| = \left|
    \PE_{x,\theta}^{(l)}\left[\sum_{j = 0}^{\tau_\Dset-1} (1-a)^{j+1}
      \bar f(X_j) \right]  \right| \\
  \leq |\bar f|_{V^\beta} \; \PE_{x,\theta}^{(l)}\left[\sum_{j =
      0}^{\tau_\Dset-1} (1-a)^{j+1} V^\beta(X_j) \right] \leq |\bar
  f|_{V^\beta} \; \frac{V^{\beta+\alpha/q}(x)}{a^{1-1/q}} \frac{\left( 1 + \bar
      b [\varepsilon \nu(\Dset)]^{-1}\right)}{(1-a)^{1/q}} (\alpha
  c)^{-1/q}\eqsp.
\end{multline*}
Applied with $\tau = \tau_\Dset$,
Propositions~\ref{prop:ComparaisonGal}(\ref{prop:CpG1} and (\ref{prop:CpG3})
and \ref{prop:TimeFiniteAS} yield
\begin{multline*}
  |\bar f|_{V^\beta}^{-1} \; \left| \sum_{j \geq 0} (1-a)^{j+1}
    \PE_{x,\theta}^{(l)}\left[\bar f(X_j) \un_{ \tau^0 \leq j < \tau^1}\right]
  \right| = |\bar f|_{V^\beta}^{-1} \; \left| \PE_{x,\theta}^{(l)}\left[
      \sum_{j = \tau_\Dset }^{\tau_\Dset +n_\star +\tau_\Dset \circ
        \underline{\theta}^{n_\star + \tau_\Dset} -1} (1-a)^{j+1} \bar f(X_j)
    \right]\right|
  \\
  \leq \PE_{x,\theta}^{(l)}\left[ \PE_{Z_{\tau_\Dset}}^{(\tau_\Dset+l)}\left[
      \sum_{j
        =0}^{n_\star+\tau_\Dset  \circ \underline{\theta}^{n_\star}-1} (1-a)^{j+1} V^\beta(X_j) \right]\right] \\
  \leq \PE_{x,\theta}^{(l)}\left[
    \PE_{Z_{\tau_\Dset}}^{(\tau_\Dset+l)}\left[\sum_{j =0}^{n_\star-1}
      (1-a)^{j+1} V^\beta(X_j) \right]\right] + \PE_{x,\theta}^{(l)}\left[
    \PE_{Z_{\tau_\Dset+n_\star}}^{(\tau_\Dset+n_\star+l)}\left[\sum_{j
        =0}^{\tau_\Dset-1} (1-a)^{j+1} V^\beta(X_j) \right]\right] \\
  \leq 2 \; \frac{(1+\bar b n_\star) (1+\bar b)}{a^{1-1/q} (1-a)^{1/q}} (\alpha
  c)^{-1/q} \ \sup_\Dset V^{\beta +\alpha/q} \eqsp.
\end{multline*}
For $k \geq 1$,
\begin{multline*}
  \left|\sum_{j \geq 0} (1-a)^{j+1} \; \PE_{x,\theta}^{(l)}\left[\bar f(X_j)
      \un_{\tau^k \leq j < \tau^{k+1}}\right] \right| \leq \left|
    \PE_{x,\theta}^{(l)}\left[\sum_{j = \tau^k}^{\tau^k + N-1} (1-a)^{j+1} \;
      \bar
      f(X_j)\right]  \right| \\
  + \PE_{x,\theta}^{(l)}\left[ (1-a)^{\tau^k+N} \;
    \PE_{Z_{\tau^k+N}}^{(\tau^k+N+l)}\left[ \sum_{j=0}^{\tau_\Dset -1}
      (1-a)^{j+1} \; \left|\bar f\right|(X_j) \right]\right]\eqsp.
\end{multline*}
By Proposition~\ref{prop:ComparaisonGal}(\ref{prop:CpG1}) and (\ref{prop:CpG2})
applied with $\tau = \tau_\Dset$, Proposition~\ref{prop:TimeFiniteAS} and Eq.
(\ref{eq:ResultPropAdaptive2}), and upon noting that $\tau^k \geq n_\star +
(k-1) N$ $\PP_{(x,\theta)}^{(l)}$-\as,
\begin{multline*}
  \left| \sum_{j \geq 0} (1-a)^{j+1} \; \PE_{x,\theta}^{(l)}\left[\bar f(X_j)
      \un_{\tau^k
        \leq j < \tau^{k+1}}\right]\right| \\
  \leq |\bar f|_{V^\beta} \; \PE_{x,\theta}^{(l)}\left[ (1-a)^{n_\star+(k-1)N}
    \; \left(
      3 N \epsilon + (1-a)^{N}  \{ V^{\beta +\alpha}(X_{\tau^k+N}) + \bar b  [\varepsilon \nu(\Dset)]^{-1}  \}  (\alpha c)^{-1}\right) \right]  \\
  \leq |\bar f|_{V^\beta} \; (1-a)^{n_\star+(k-1)N} \; \left( 3 N \epsilon +
    (\alpha c)^{-1} \ \sup_{r,\Dset \times \Theta} \PE_{x,\theta}^{(r)}\left[
      V^{\beta
        +\alpha}(X_{N}) + \bar b [\varepsilon \nu(\Dset)]^{-1}  \right]\right) \\
  \leq |\bar f|_{V^\beta} \; (1-a)^{n_\star+(k-1)N} \; \left( 3 N \epsilon +
    (\alpha c)^{-1} \left(  \sup_{\Dset} \ V^{\beta +\alpha} + \bar b N^{\beta +\alpha}+ \bar b [\varepsilon \nu(\Dset)]^{-1}   \right)\right) \\
  \leq 4 \; \epsilon \; |\bar f|_{V^\beta} \; (1-a)^{(k-1)N} \; N \eqsp,
\end{multline*}
where we used the definition of $N$ (see Eq.~(\ref{eq:Controle4})) and
Proposition~\ref{prop:ComparaisonGal}(\ref{prop:CpG1}). This yields the desired
result.

\paragraph{\tt  Proof of Eq.(\ref{eq:ResultPropAdaptive2})}
By the strong Markov property and since $\tau^k \geq n_\star + N(k-1)$
$\PP_{x,\theta}^{(l)}$-\as
\begin{multline*}
  \left| \PE_{x,\theta}^{(l)} \left[ \sum_{j=0}^{N-1} (1-a)^{\tau^k+ j+1} \;
      \bar f \left( X_{\tau^{k}+j} \right)\right] \right| \leq (1-a)^{n_\star+
    N(k-1)} \PE_{x,\theta}^{(l)} \left[ \left|\PE_{Z_{\tau^{k}}}^{(\tau^{k}+l)}
      \left[ \sum_{j=0}^{N-1} (1-a)^{j+1} \; \bar f(X_j) \right]\right|\right].
\end{multline*}
Furthermore, by Proposition~\ref{prop:ContructionCouplingOpt},
\begin{multline*}
  \PE_{Z_{\tau^{k}}}^{(\tau^{k}+l)} \left[ \sum_{j=0}^{N-1} (1-a)^{j+1} \; \bar
    f(X_j) \right] = \bPE_{Z_{\tau^k}, Z_{\tau^{k}}}^{(\tau^{k}+l)} \left[
    \sum_{j=0}^{N-1} (1-a)^{j+1} \; \bar f(X_j) \right]
  \\
  = \bPE_{Z_{\tau^{k}}, Z_{\tau^{k}}}^{(\tau^{k}+l)} \left[ \sum_{j=0}^{N-1}
    (1-a)^{j+1} \; \bar f(\tilde X_j) \right] \ + \bPE_{Z_{\tau^{k}},
    Z_{\tau^{k}}}^{(\tau^{k}+l)} \left[ \sum_{j=0}^{N-1} (1-a)^{j+1} \; \{ \bar
    f( X_j) - \bar f(\tilde X_j) \} \un_{\mathcal{E}^c_N}\right].
\end{multline*}
On one hand, we have $ \PP_{x,\theta}^{(l)}-\as$,
\begin{multline*}
  \left|\bPE_{Z_{\tau^{k}}, Z_{\tau^{k}}}^{(\tau^{k}+l)} \left[
      \sum_{j=0}^{N-1} (1-a)^{j+1} \; \bar f(\tilde X_j) \right] \right| \leq
  |\bar f|_{V^\beta} \; \sum_{j=0}^{N-1} (1-a)^{j+1} \; \sup_{\Dset \times \Theta}
  \; \| P_{\theta}^j(x,\cdot) -\pi(\cdot) \|_{V^\beta} \leq |\bar f|_{V^\beta}
  \; N \epsilon
\end{multline*}
by (\ref{eq:Controle3}).  On the other hand, $ \PP_{x,\theta}^{(l)}-\as$,
\begin{multline*}
  \left| \bPE_{Z_{\tau^{k}}, Z_{\tau^{k}}}^{(\tau^{k}+l)} \left[
      \sum_{j=0}^{N-1} (1-a)^{j+1} \; \{ \bar f( X_j) - \bar f(\tilde X_j) \}
      \un_{\mathcal{E}^c_N}\right] \right| \\
  \leq |\bar f|_{V^\beta} \; \bPE_{Z_{\tau^{k}},Z_{\tau^{k}}}^{(\tau^{k}+l)}
  \left[ \sum_{j=0}^{N-1} (1-a)^{j+1} \; \{ V^\beta( X_j) + V^\beta(\tilde X_j)
    \}
    \un_{\mathcal{E}^c_N}\right] \\
  \leq |\bar f|_{V^\beta} \; \bPE_{Z_{\tau^{k}}, Z_{\tau^{k}}}^{(\tau^{k}+l)}
  \left[ \left( \sum_{j=0}^{N-1} (1-a)^{j+1} \; \left\{ V^\beta ( X_j) +
        V^\beta(\tilde X_j) \right\} \right)^{\beta^{-1}}\right]^{\beta} \left(
    \bPP_{Z_{\tau^{k}}, Z_{\tau^{k}}}^{(\tau^{k}+l)} \left(\mathcal{E}^c_N \right)
  \right)^{1-\beta}
\end{multline*}
by using the Jensen's inequality ($\beta <1$).  By the Minkowski inequality, by
Proposition~\ref{prop:ComparaisonGal}(\ref{prop:CpG1}), and by iterating the
drift inequality A\ref{A5}
\begin{multline*}
  \bPE_{Z_{\tau^{k}}, Z_{\tau^{k}}}^{(\tau^{k}+l)} \left[ \left(
      \sum_{j=0}^{N-1} (1-a)^{j+1} \; \left\{ V^\beta ( X_j) + V^\beta(\tilde
        X_j) \right\}
    \right)^{\beta^{-1}}\right]^{\beta} \\
  \leq \sum_{j=0}^{N-1} (1-a)^{j+1} \; \left\{ \bPE_{Z_{\tau^{k}},
      Z_{\tau^{k}}}^{(\tau^{k}+l)} \left[ V ( X_j) \right]^\beta +
    \bPE_{Z_{\tau^{k}}, Z_{\tau^{k}}}^{(\tau^{k}+l)} \left[ V ( \tilde X_j)
    \right]^\beta \right\}
  \\
  \leq \sum_{j=0}^{N-1} (1-a)^{j+1} \; \left\{ \sup_{l, \Dset \times \Theta}
    \left(\PE_{x,\theta}^{(l)} \left[ V(X_j)\right]\right)^\beta + \left(
      \sup_{\Dset \times \Theta} P^j_{\theta} V(x)\right)^\beta \right\}
  \\
  \leq 2 \; \sum_{j=0}^{N-1} (1-a)^{j+1} \left( \sup_\Dset V +j \bar b
  \right)^\beta \leq 2 N \left( \sup_\Dset V^\beta + \bar b N^{-1}
    \sum_{j=1}^{N-1} j^\beta\right) \eqsp.
\end{multline*}
Finally,
\[
\PE_{x,\theta}^{(l)} \left[ \left( \bPP_{Z_{\tau^k},
      Z_{\tau^k}}^{(\tau^k+l)}(\mathcal{E}^c_N) \right)^{1-\beta} \right] \leq
\left(\PE_{x,\theta}^{(l)} \left[ \bPP_{Z_{\tau^k},
      Z_{\tau^k}}^{({\tau^k}+l)}(\mathcal{E}^c_N) \right]\right)^{1-\beta} \leq
\left( N^2 \epsilon_N \right)^{1-\beta}
\]
where we used (\ref{eq:CouplingProbability2}) in the last inequality.  To
conclude the proof, use the definition of $\epsilon_N$.

\end{proof}

\subsection{Proof of Theorem~\ref{theo:MarginalUnifCase}}
Let $\epsilon >0$.  We prove that there exists $n_\epsilon$ such that for any
$n\geq n_\epsilon$, $\sup_{\{f, |f|_1 \leq 1\}} \left|\PE_{\xi_1,\xi_2}\left[
    \bar f(X_n) \right] \right| \leq \epsilon$.

\subsubsection{Definition of $\Dset$,  $N$, $Q$ and $n_\star$}
By A\ref{A-VCset}(\ref{Anew}), choose $Q$ such that
\begin{equation}
  \label{eq:DefinitionL}
\sup_l \sup_{(x,\theta) \in \Cset \times \Theta}  \PE_{x,\theta}^{(l)} \left[ \rate(\tau_\Cset) \right] \  \sum_{k \geq Q} \frac{1}{\rate(k)} \leq  \epsilon \eqsp.
\end{equation}
By A\ref{A-VCset}(\ref{A4rev}), choose $N$ such that
\begin{equation}
  \label{eq:DefinitionN}
  \sup_{(x, \theta) \in \Cset \times \Theta} V^{-1}(x) \ \| P_\theta^N(x,\cdot) - \pi(\cdot) \|_{\tv} \leq  \frac{\epsilon}{Q} \eqsp.
\end{equation}
By B\ref{B1}, choose $n_\star$ such that for any $n \geq n_\star$,
\begin{equation}
  \label{eq:DefiNstar}
  \PP_{\xi_1,\xi_2} \left( D(\theta_n, \theta_{n-1}) \geq \epsilon /(2 (N+Q-1)^2  Q)\right) \leq \frac{\epsilon}{4(N+Q-1)^2  Q } \eqsp.
\end{equation}

\subsubsection{Optimal coupling}
We apply Proposition~\ref{prop:ContructionCouplingOpt} with $l =0$ and $N
\leftarrow N+Q$.  Set $\mathcal{E}_{N +Q} \eqdef \{X_k = \tilde X_k, 0 \leq k
\leq N +Q \}$. It holds for any $ r \geq n_\star$,
\begin{multline}
\label{eq:CouplingProb}
\PE_{\xi_1,\xi_2} \left[ \un_{X_r \in \Cset} \; \bPP_{Z_{r}, Z_r}^{(r)} \left(
    \mathcal{E}_{N+Q}^c \right) \right] \leq \sum_{j=1}^{N+Q-1} \sum_{i=1}^j
\PE_{\xi_1,\xi_2}\left[\un_{X_r \in \Cset} \; \PE_{Z_r}^{(r)}\left[D(\theta_i,
    \theta_{i-1}) \right] \right]  \\
\leq \sum_{j=1}^{N+Q-1} \sum_{i=1}^j \PE_{\xi_1,\xi_2}\left[ D(\theta_{i+r},
  \theta_{i+r-1}) \right] \leq \epsilon Q^{-1} \eqsp,
\end{multline}
where in the last inequality, we use that $D(\theta,\theta') \leq 2$ and the
definition of $n_\star$ (see Eq.~(\ref{eq:DefiNstar})).

\subsubsection{Proof}
Let $n \geq N+Q+n_\star$. We consider the partition given by the last exit from
the set $\Cset$ before time $n-N$. We use the notation $\{X_{n:m} \notin \Cset
\}$ as a shorthand notation for $\bigcap_{k=n}^m \{X_k \notin \Cset \}$; with
the convention that $\{X_{m+1:m} \notin \Cset \} = \Omega$. We write
\begin{multline*}
  \PE_{\xi_1,\xi_2} \left[ \bar f(X_n) \right] = \PE_{\xi_1,\xi_2} \left[ \bar
    f(X_n) \un_{X_{0:n-N} \notin \Cset} \right] + \sum_{k=0}^{n-N}
  \PE_{\xi_1,\xi_2} \left[ \bar f(X_n) \un_{X_k \in \Cset } \ \un_{X_{k+1:n-N}
      \notin \Cset} \right] \eqsp.
\end{multline*}

Since $\bar f$ is bounded on $\Xset$ by $|\bar f|_1$, we have
\[
\PE_{\xi_1,\xi_2} \left[ \bar f(X_n) \un_{X_{0:n-N} \notin \Cset} \right] \leq
|\bar f|_1 \ \PP_{\xi_1, \xi_2} \left( \tau_\Cset \geq n-N\right) \leq |\bar
f|_1 \ \PE_{\xi_1, \xi_2} \left[ \frac{\tau_\Cset }{n-N} \wedge 1 \right] \eqsp.
\]
The rhs is upper bounded by $|\bar f|_1 \ \epsilon$ for $n$ large enough. By
definition of $Q$ in (\ref{eq:DefinitionL}),
\begin{multline}
\label{eq:weaken}
  \sum_{k=0}^{n-(N+Q)} \PE_{\xi_1,\xi_2} \left[ \bar f(X_n) \un_{X_k \in \Cset
    } \ \un_{X_{k+1:n-N} \notin \Cset} \right] \leq |\bar f|_1 \
  \sum_{k=0}^{n-(N+Q)} \PE_{\xi_1,\xi_2} \left[ \un_{X_k \in \Cset }
    \PP_{X_k,\theta_k}^{(k)} \left( \tau_\Cset \geq
      n-N-k\right) \ \right] \\
  \leq |\bar f|_1 \ \sup_l \sup_{\Cset \times \Theta} \PE_{x,\theta}^{(l)} \left[
    \rate(\tau_\Cset) \right] \sum_{k \geq Q} \frac{1}{\rate(k)} \leq |\bar
  f|_1 \ \epsilon \eqsp.
\end{multline}
Let $k \in \{n-(N+Q)+1, \cdots, n-N \}$. By definition of $N$ and $n_\star$
(see Eqs.~(\ref{eq:DefinitionN}) and (\ref{eq:DefiNstar})), upon noting that $k
\geq n-(N+Q) \geq n_\star$,
\begin{multline*}
  \PE_{\xi_1,\xi_2} \left[ \bar f(X_n) \un_{X_k \in \Cset } \ \un_{X_{k+1:n-N}
      \notin \Cset} \right] - |\bar f|_1 \ \PE_{\xi_1,\xi_2} \left[ \un_{X_k
      \in \Cset} \ \bPP_{Z_k,Z_k}^{(k)} \left(
      \mathcal{E}_{N+Q}^c\right) \right] \\
  \leq \PE_{\xi_1,\xi_2} \left[ \un_{X_k \in \Cset} \ \bPE_{Z_k,Z_k}^{(k)}
    \left[ \bar f(X_{n-k}) \un_{X_{1:n-N-k} \notin \Cset}
      \un_{\mathcal{E}_{N+Q}} \right] \right]  \\
  \leq \PE_{\xi_1,\xi_2} \left[ \un_{X_k \in \Cset} \ \bPE_{Z_k,Z_k}^{(k)}
    \left[ \bar f(\tilde X_{n-k}) \un_{\tilde X_{1:n-N-k} \notin \Cset}
      \un_{\mathcal{E}_{N+Q}} \right] \right] \\
  \leq \PE_{\xi_1,\xi_2} \left[ \un_{X_k \in \Cset} \ \bPE_{Z_k,Z_k}^{(k)}
    \left[ \bar f(\tilde X_{n-k}) \un_{\tilde X_{1:n-N-k} \notin \Cset} \right]
  \right] + |\bar f|_1 \ \PE_{\xi_1,\xi_2} \left[ \un_{X_k \in \Cset} \
    \bPP_{Z_k,Z_k}^{(k)} \left(\mathcal{E}_{N+Q}^c \right) \right] \\
  \leq \PE_{\xi_1,\xi_2} \left[ \un_{X_k \in \Cset} \ \bPE_{Z_k,Z_k}^{(k)}
    \left[ \un_{\tilde X_{1:n-N-k} \notin \Cset} P_{\theta_k}^{N}\bar f(\tilde
      X_{n-N-k}) \right] \right] + |\bar f|_1 \ \epsilon Q^{-1}  \\
  \leq |\bar f|_1 \ \epsilon Q^{-1} \PE_{\xi_1,\xi_2} \left[ \un_{X_k \in
      \Cset} \ \bPE_{Z_k,Z_k}^{(k)}\left[ \un_{\tilde X_{1:n-N-k} \notin \Cset}
      V(\tilde X_{n-N-k}) \right] \right] + |\bar f|_1 \ \epsilon Q^{-1}
  \\
  \leq |\bar f|_1 \ \epsilon Q^{-1} \left\{ \sup_{(x,\theta) \in \Cset \times
      \Theta} P_\theta V(x) + \sup_\Cset V \right\} + |\bar f|_1 \ \epsilon
  Q^{-1}  \eqsp,
\end{multline*}
where we used A\ref{A-VCset}(\ref{A3rev}) in the last inequality.  Hence,
\[
\sum_{k=n-(N+Q)+1}^{n-N} \PE_{\xi_1,\xi_2} \left[ \bar f(X_n) \un_{X_k \in
    \Cset } \ \un_{X_{k+1:n-N} \notin \Cset} \right] \leq \left(1 +
  \sup_{(x,\theta) \in \Cset \times \Theta} P_\theta V(x) + \sup_\Cset V \right) \epsilon \
|\bar f|_1 \eqsp.
\]
This concludes the proof.

\begin{rem}
  \label{rem:YanBai}
  In the case the process is non-adaptive, we can assume w.l.g. that it
  possesses an atom $\alpha$; in that case, the lines (\ref{eq:weaken}) can be
  modified so that the assumptions $\sum_n \{1/\rate(n) \}<+\infty$ can be
  removed. In the case of an atomic chain, we can indeed apply the above
  computations with $\Cset$ replaced by $\alpha$ and write:
  \begin{multline*}
    \sum_{k=0}^{n-(N+Q)} \PE_{\xi_1} \left[ \bar f(X_n) \un_{X_k \in \alpha } \
      \un_{X_{k+1:n-N} \notin \alpha} \right] \leq |\bar f|_1 \
    \sum_{k=0}^{n-(N+Q)} \PP_{\alpha} \left( \tau_\alpha \geq n-N-k\right) \\
    \leq |\bar f|_1 \ \sum_{k \geq Q} \PP_{\alpha} \left( \tau_\alpha \geq
      k\right) \eqsp.
\end{multline*}
The rhs is small for convenient $Q$, provided
$\PE_\alpha[\rate(\tau_\alpha)]<+\infty$ with $\rate(n) =n$.  Unfortunately,
the adaptive chain $\{(X_n, \theta_n), n\geq 0\}$ does not possess an atom thus
explaining the condition on $\rate$.
\end{rem}

\subsection{Proof of Corollary~\ref{coro:MarginalUnifCase}}
The condition A\ref{A-VCset}(\ref{A4rev}) is established in
Appendix~\ref{app:UniformControl}.  Let a level set $\Dset$ large enough such
that $\nu(\Dset) >0$; then Proposition~\ref{prop:TimeFiniteAS} implies that
there exists a constant $c < \infty$ such that for any $l \geq 0$,
$\PE_{x,\theta}^{(l)}\left[ \tau_\Dset \right] \leq c V(x)$. This implies that
for $0< \eta \leq 1- \alpha$,
\begin{multline*}
  \PE_{x,\theta}^{(l)}\left[ \sum_{k=0}^{\tau_\Dset} (k+1)^{\eta} \right] \leq
  \PE_{x,\theta}^{(l)}\left[ \sum_{k=0}^{\tau_\Dset} \left(\PE_{X_k, \theta_k}^{(k+l)}
      \left[ \tau_\Dset \right] \right)^{\eta} \right] \leq c^\eta \
  \PE_{x,\theta}^{(l)}\left[ \sum_{k=0}^{\tau_\Dset} V^{1-\alpha}(X_k) \right] \\
  \leq C \ \left( V(x) + b \ \PE_{x,\theta}^{(l)}\left[ \tau_\Dset \right] \right)
  \leq C' \ V(x) \eqsp,
\end{multline*}
for some finite constants $C,C'$ independent upon $\theta$. Hence
A\ref{A-VCset}(\ref{Anew}) holds with $\rate(n) \sim n^{1+\eta}$. Finally,
$P_\theta V \leq V - c V^{1-\alpha} +b \un_\Cset$ implies $P_\theta V \leq V - c
\gamma V^{1-\alpha} + b \un_\Dset$ for any $\gamma \in (0,1)$ and the level set
$\Dset \eqdef \{x, V^{1-\alpha} \leq b [c(1-\gamma)]^{-1} \}$. This yields A\ref{A-VCset}(\ref{A3rev}).

\subsection{Proof of Proposition~\ref{prop:YanBai}}
Under A\ref{Adrift}, there exists a constant $C$ - that does not depend upon
$\theta$ - such that for any $(x,\theta) \in \Xset \times \Theta$, $n\geq 0$
and $\kappa \in [1, \alpha^{-1}]$,
\[
\ \| P^n_\theta(x,\cdot) - \pi(\theta) \|_\tv \leq C \frac{
  V^{\kappa \alpha}(x)}{(n+1)^{\kappa-1}} \eqsp;
\]
(see Appendix~\ref{app:UniformControl}). To apply \cite[Theorem
13]{rosenthaletroberts05}, we only have to prove that there exists $\kappa \in
[1, \alpha^{-1}]$ such that the sequence $\{V^{\kappa \alpha}(X_n); n\geq 0\}$
is bounded in probability, which is equivalent to prove that $\{V^\beta(X_n);
n\geq 0\}$ is bounded in probability for some (and thus any) $\beta \in (0,1]$
. This is a consequence of Lemma~\ref{lem:YanBai} applied with $W = V^\beta$
for some $\beta \in (0,1]$ and $\rate(n) = (n+1)^{1+\eta}$ for some $\eta>0$
(see the proof of Corollary~\ref{coro:MarginalUnifCase} for similar computations).

\begin{lemma}
\label{lem:YanBai}
Assume that there exist a set $\Cset$ and functions $W: \Xset \to (0, +\infty)$
and $\rate : \nset \to (0, +\infty)$ such that $\rate$ is non-decreasing,
$P_\theta W \leq W $ on $\Cset^c$ and
\[
\sup_{\Cset \times \Theta} P_\theta W< +\infty \eqsp, \qquad \qquad \sup_l
\sup_{\Cset\times \Theta } \PE_{x,\theta}^{(l)}\left[ \rate(\tau_\Cset) \right]
< +\infty \eqsp, \qquad \qquad \sum_k \{1/\rate(k) \} < +\infty
\]
For any probability distributions $\xi_1, \xi_2$ resp. on $\Xset,\Theta$
$\{W(X_n), n\geq 0 \}$ is bounded in probability for the probability
$\PP_{\xi_1,\xi_2}$.
\end{lemma}
\begin{proof}
  Let $\epsilon >0$. We prove that there exists $M_\epsilon, N_\epsilon$ such
that for any $M \geq M_\epsilon$ and $n \geq N_\epsilon$, $\PP_{x,\theta}\left(
  W(X_n) \geq M \right) \leq \epsilon$. Choose $N_\epsilon$ such that for any
$n \geq N_\epsilon$
\[
\PE_{\xi_1,\xi_2}\left[\frac{ \tau_{\Cset}}{n} \wedge 1 \right] \leq
\epsilon/3\eqsp, \qquad \qquad \sup_l \sup_{\Cset \times \Theta}
\PE_{x,\theta}^{(l)}\left[ \rate(\tau_\Cset) \right]\ \sum_{k \geq n}
\{1/\rate(k) \} \leq \epsilon/3 \eqsp,
\]
and choose $M_\epsilon $ such that for any $M \geq M_\epsilon$,
$ N_\epsilon \ \sup_{\Cset \times \Theta} P_\theta W \leq \epsilon M /3$.   We write
\[
\PP_{\xi_1,\xi_2}\left( W(X_n) \geq M \right) = \sum_{k=0}^{n-1} \PP_{\xi_1,\xi_2}
\left( W(X_n) \geq M,X_k \in \Cset, X_{k+1:n} \notin \Cset \right) +
\PP_{\xi_1,\xi_2} \left( W(X_n) \geq M, X_{0:n} \notin \Cset \right) \eqsp.
\]
By the Markov inequality, for $n \geq N_\epsilon$,
\[
\PP_{\xi_1,\xi_2}\left( W(X_n) \geq M , X_{0:n} \notin \Cset \right) \leq
\PP_{\xi_1,\xi_2}\left( X_{0:n} \notin \Cset \right) \leq
\PP_{\xi_1,\xi_2}\left( \tau_\Cset > n \right) \leq
\PE_{\xi_1,\xi_2}\left[\frac{ \tau_{\Cset}}{n} \wedge 1 \right] \leq \epsilon/3
\eqsp.
\]
Furthermore,  for $n \geq N_\epsilon$,
\begin{multline*}
  \sum_{k=0}^{n-N_\epsilon} \PP_{\xi_1,\xi_2} \left( W(X_n) \geq M,X_k \in
    \Cset, X_{k+1:n} \notin \Cset \right) \leq \sum_{k=0}^{n-N_\epsilon}
  \PP_{\xi_1,\xi_2} \left(X_k \in \Cset, X_{k+1:n} \notin \Cset \right) \\
  \leq \sum_{k=0}^{n-N_\epsilon} \PE_{\xi_1,\xi_2} \left[ \un_\Cset(X_k) \
    \sup_{l} \sup_{\Cset \times \Theta} \PP_{x,\theta}^{(l)}\left( X_{1:n-k}
      \notin \Cset \right)\right] \leq \sum_{k=0}^{n-N_\epsilon} \sup_l
  \sup_{\Cset \times \Theta} \PP_{x,\theta}^{(l)} \left( \tau_{\Cset} \geq n-k
  \right) \\
  \leq \sum_{k=N_\epsilon}^{n} \frac{1}{\rate(k)} \sup_l \sup_{\Cset \times
    \Theta} \PE_{x,\theta}^{(l)}\left[ \rate(\tau_\Cset) \right] \leq \epsilon/3
  \eqsp.
\end{multline*}
Finally, for $n \geq N_\epsilon$ we write
\begin{multline*}
  \sum_{k=n-N_\epsilon+1}^n \PP_{x,\theta} \left( W(X_n) \geq M,X_k \in \Cset,
    X_{k+1:n} \notin \Cset \right) \\
  \leq \sum_{k=n-N_\epsilon+1}^n \PE_{x,\theta} \left[ \un_\Cset(X_k) \
    \PP_{X_k,\theta_k}^{(k)}\left(W(X_{n-k}) \geq M, X_{1:n-k} \notin \Cset \right)
  \right]
\end{multline*}
We have, for any $k \in \{n-N_\epsilon+1, \cdots, n \}$ and $(x,\theta) \in
\Cset \times \Theta$
\begin{multline*}
  \PP_{x,\theta}^{(k)}\left(W(X_{n-k}) \geq M, X_{1:n-k} \notin \Cset \right)
  \leq \frac{1}{M} \PE_{x,\theta}^{(k)}\left[ W(X_{n-k}) \un_{\Cset^c}(
    X_{1:n-k-1}) \right] \leq \frac{1}{M} \PE_{x,\theta}^{(k)}\left[ W(X_1)
  \right]
\end{multline*}
where, in the last inequality, we used the drift inequality on $W$ outside
$\Cset$.  Hence,
\[
\sum_{k=n-N_\epsilon+1}^n \PP_{x,\theta} \left( W(X_n) \geq M,X_k \in \Cset,
  X_{k+1:n} \notin \Cset \right) \leq \frac{N_\epsilon}{M} \sup_{\Cset \times
  \Theta} P_\theta W(x) \leq \epsilon /3 \eqsp.
\]
The proof is concluded.
\end{proof}

\subsection{Proof of Theorem~\ref{theo:SLLNUnboundedUnifCase}}

By using the function $\hat{g}_a^{(l)}$ introduced in
Section~\ref{sec:GeneralPoisson} and by Proposition~\ref{prop:QuasiPoissonEq},
we write $\PP_{x,\theta}-\as$
\begin{multline*}
  n^{-1} \sum_{k=1}^n \bar f(X_k) = n^{-1} \sum_{k=1}^n \left( (1-a)^{-1}\hat
    g_a^{(k)}(X_k,\theta_k) -\PE_{X_k,\theta_k}^{(k)} \left[ \hat
      g_a^{(k+1)}\left(X_1, \theta_1 \right) \right] \right) \\
  = n^{-1} (1-a)^{-1} \; \sum_{k=1}^n \left\{ \hat g_a^{(k)}(X_k,\theta_k) -
    \PE_{x,\theta}\left[\hat g_a^{(k)}(X_k,\theta_k) \vert \mathcal{F}_{k-1}
    \right]
  \right\} \\
  + n^{-1} (1-a)^{-1} \sum_{k=1}^n \left\{ \PE_{x,\theta}\left[\hat
      g_a^{(k)}(X_k,\theta_k) \vert \mathcal{F}_{k-1} \right] - (1-a)
    \PE_{x,\theta} \left[ \hat g_a^{(k+1)}\left(X_{k+1}, \theta_{k+1} \right)
      \vert \mathcal{F}_{k} \right]
  \right\} \\
  = n^{-1} (1-a)^{-1} \; \sum_{k=1}^n \left\{ \hat g_a^{(k)}(X_k,\theta_k) -
    \PE_{x,\theta}\left[\hat g_a^{(k)}(X_k,\theta_k) \vert \mathcal{F}_{k-1}
    \right]
  \right\} \\
  + n^{-1} (1-a)^{-1} \left\{ \PE_{x,\theta}\left[\hat g_a^{(1)}(X_1,\theta_1)
      \vert \mathcal{F}_{0} \right] - \PE_{x,\theta}\left[\hat
      g_a^{(n+1)}(X_{n+1},\theta_{n+1}) \vert
      \mathcal{F}_{n} \right] \right\} \\
  + a \; n^{-1} (1-a)^{-1} \; \sum_{k=1}^n \PE_{x,\theta} \left[ \hat
    g_a^{(k+1)}\left(X_{k+1}, \theta_{k+1} \right) \vert \mathcal{F}_{k}
  \right].
\end{multline*}
We apply the above inequalities with $a = a_n$ and consider the different terms
in turn. We show that they tend $\PP_{x,\theta}-\as$ to zero when the
deterministic sequence $\{a_n, n \geq 1 \}$ satisfies conditions which are
verified e.g. with $a_n = (n+1)^{-\zeta}$ for some $\zeta$ such that
\[
\zeta >0 \eqsp, \qquad 2\zeta < 1 - \left(0.5 \vee \beta(1-\alpha)^{-1} \right)
\eqsp, \qquad \zeta < 1 - \beta (1-\alpha)^{-1} \eqsp.
\]
To prove that each term converges a.s. to zero, we use the following
characterization
\[
\left[ \forall \epsilon>0 \eqsp, \quad \lim_{n \to +\infty} \PP\left(\sup_{m
      \geq n} |X_m| \geq \epsilon \right) \right] \Longleftrightarrow \left[
  \{X_n, n \geq 0 \} \to 0 \qquad \PP-\as \right] \eqsp.
\]

Hereafter, we assume that $|f|_{V^\beta} =1$. In the following, $c$ (and below, $c_1,c_2$) are constant the value of which may vary upon each appearance.

\paragraph{\tt Convergence of  Term 1.} Set $p \eqdef (1-\alpha) /\beta$.
We prove that
\[
n^{-1} (1-a_n)^{-1} \sum_{k=1}^n \left\{ \hat g_{a_n}^{(k)}(X_k,\theta_k) -
  \PE_{\xi_1,\xi_2}\left[\hat g_{a_n}^{(k)}(X_k,\theta_k) \vert
    \mathcal{F}_{k-1} \right] \right\} \longrightarrow 0 \eqsp,
\PP_{\xi_1,\xi_2}-\as
\]
provided the sequence $\{a_n, n\geq 0\}$ is non increasing, $ \lim_{n\to\infty}
\ n^{\max(1/p,1/2)-1} /a_n = 0$, $\sum_n n^{-1} [n^{\max(1/p,1/2)-1} /a_n]^p <
+\infty$ and $\sum_n |a_n -a_{n-1}| a_{n-1}^{-2} \  [n^{\max(1/p,1/2)-1} /a_n] <
+\infty$.
\begin{proof}
  Define $D_{n,k} \eqdef \hat g_{a_n}^{(k)}(X_k,\theta_k) -
  \PE_{\xi_1,\xi_2}\left[\hat g_{a_n}^{(k)}(X_k,\theta_k) \vert
    \mathcal{F}_{k-1} \right]$; $S_{n,k} \eqdef \sum_{j=1}^kD_{n,j}$, if $k\leq
  n$ and $S_{n,k} \eqdef \sum_{j=1}^nD_{n,j}+\sum_{j=n+1}^kD_{j,j}$ if $k>n$;
  and $R_{n} \eqdef \sum_{j=1}^{n-1}D_{n,j}-D_{n-1,j}$. Then for each $n$,
  $\{(S_{n,k},\F_k),\;k\geq 1\}$ is a martingale. For $k>n$ and by Lemma
  \ref{lem1martingales}, there exists a universal constant $C$ such that
\begin{multline}\PE_{\xi_1,\xi_2}\left[|S_{n,k}|^p\right]\leq Ck^{\max(p/2,1)-1}\left(\sum_{j=1}^n\PE_{\xi_1,\xi_2}\left[|D_{n,j}|^p\right]+\sum_{j=n+1}^k\PE_{\xi_1,\xi_2}\left[|D_{j,j}|^p\right]\right)\\
  \leq c_1 \ |\bar f|_{V^\beta} \ 
  k^{\max(p/2,1)-1}a_k^{-p}\sum_{j=1}^k\PE_{\xi_1,\xi_2}\left[V(X_j)\right]\leq
  c_1 \ |\bar f|_{V^\beta} \ k^{\max(p/2,1)}a_k^{-p}\xi_1(V),
  \label{eq:BoundOnSnk} \end{multline} where we used (\ref{eq:MajoRem}) and
Proposition~\ref{prop:ComparaisonGal}(\ref{prop:CpG2}). It follows that for any
$n\geq 1$, $\lim_{N\to\infty}
N^{-p}\PE_{\xi_1,\xi_2}\left(|S_{n,N}|^p\right)\leq
c_1\lim_{N\to\infty}\left(N^{\max(1/p,1/2)-1} /a_N\right)^p=0$. Then by the
martingale array extension of the Chow-Birnbaum-Marshall's inequality
(Lemma~\ref{lem:Birnbaum}),
\begin{multline*}
  2^{-p}\delta^p\PP_{\xi_1,\xi_2}\left(\sup_{m \geq n} m^{-1} (1-a_m)^{-1} \left|\sum_{j=1}^nD_{n,j}\right|>\delta\right)\\
  \leq
  \sum_{k=n}^\infty\left(k^{-p}-(k+1)^{-p}\right)\PE_{\xi_1,\xi_2}\left[|S_{n,k}|^p\right]+
  \left(\sum_{k=n+1}^\infty
    k^{-1}\PE_{\xi_1,\xi_2}^{1/p}\left[|R_k|^p\right]\right)^p
  \eqsp.
\end{multline*} Under the assumptions on the sequence $\{a_n, n\geq 0\}$  and given the bound (\ref{eq:BoundOnSnk}), the first term in the rhs tends to zero as $n \to +\infty$.
To bound the second term, we first note that
$\{(\sum_{j=1}^kD_{n,j}-D_{n-1,j},\F_k),\;k\geq 1\}$ is a martingale for each
$n$. Therefore, by Lemma \ref{lem1martingales} and the definition of $D_{n,j}$
\begin{multline*}
  \PE_{\xi_1,\xi_2}\left[|R_n|^p\right]\leq
  C \ n^{\max(p/2,1)-1}\sum_{j=1}^{n-1}\PE_{\xi_1,\xi_2}\left[|D_{n,j}-D_{n-1,j}|^p\right] \\
  \leq 2 C \ n^{\max(p/2,1)-1}\sum_{j=1}^{n-1}\PE_{\xi_1,\xi_2}\left[|\hat
    g_{a_n}^{(j)}(X_j,\theta_j) -\hat g_{a_{n-1}}^{(j)}(X_j,\theta_j)|^p\right]
  \eqsp.\end{multline*} Then, using (\ref{bounddiffGa}) (with $q=\infty$) and
the usual argument of bounding moments of $V^\beta(X_j)$, we get
\[\PE_{\xi_1,\xi_2}^{1/p}\left[|R_n|^p\right]\leq c_1 \ |\bar f|_{V^\beta} \ n^{\max(1/2,1/p)} \ |a_n - a_{n-1}|  \ a_n^{-1} a_{n-1}^{- 2}\xi_1(V).\]
Under the assumptions, $\sum_n n^{-1}
\PE_{\xi_1,\xi_2}^{1/p}\left[|R_n|^p\right] < +\infty$ and this concludes the proof.

\end{proof}

\paragraph{\tt Convergence of Term 2.}
We prove that
\[
n^{-1} (1-a_n)^{-1} \PE_{\xi_1,\xi_2}\left[\hat g_{a_n}^{(1)}(X_1,\theta_1)
  \vert \mathcal{F}_{0} \right] \longrightarrow 0 \eqsp,
\]
provided $\lim_n n a_n = +\infty$ and $\lim_n a_n =0$.
\begin{proof}
  By Theorem~\ref{theo:controleG} applied with $q= +\infty$, it may be
  proved that there exist constants $c,N$ such that
\[ \left|
  \PE_{\xi_1,\xi_2}\left[\hat g_{a_n}^{(1)}(X_1,\theta_1) \vert \mathcal{F}_{0}
  \right] \right| \leq c a_n^{-1} \xi_1(V) + c \left(1-(1-a_n)^N \right)^{-1} N
\]
Divided by $n^{-1} (1-a_n)$, the rhs tends to zero as $n \to +\infty$.
\end{proof}

\paragraph{\tt Convergence of  Term 3.}
We prove that
\[
n^{-1} (1-a_n)^{-1} \PE_{\xi_1,\xi_2}\left[\hat
  g_{a_n}^{(n+1)}(X_{n+1},\theta_{n+1}) \vert \mathcal{F}_{n} \right]
\longrightarrow 0 \eqsp, \PP_{\xi_1,\xi_2}-\as
 \]
provided the sequence $\{n^{-1} a_n^{-1}, n\geq 1 \}$ is non-increasing, $\lim_n
n^{1-\beta(1-\alpha)^{-1}} a_n = +\infty$, $\sum_n (n a_n)^{-(1-\alpha)\beta^{-1}}
< +\infty$ and $\lim_n a_n =0$.
\begin{proof}
  There exist constants $c_1,c_2,N$ such that for any $n$ large enough (i.e.
  such that $1-a_n \geq 1/2$) and $p \eqdef (1-\alpha) \beta^{-1} >1$
    \begin{multline*}
      \PP_{\xi_1,\xi_2} \left( \sup_{m \geq n} m^{-1} (1-a_m)^{-1} \; \left|
          \PE_{\xi_1,\xi_2}\left[ \hat g_{a_m}^{(m+1)}(X_{m+1},\theta_{m+1}) \vert
            \mathcal{F}_{m} \right] \right| \geq \delta \right) \\
      \leq 2^p \delta^{-p} \; \PE_{\xi_1,\xi_2}\left[ \sup_{m \geq n} m^{-p}
        \left| \PE_{\xi_1,\xi_2}\left[ \hat g_{a_m}^{(m+1)}(X_{m+1},\theta_{m+1})
            \vert
            \mathcal{F}_{m} \right] \right|^p \right] \\
      \leq 2^p \delta^{-p} \; \sum_{m \geq n} m^{-p} \; \PE_{\xi_1,\xi_2}\left[
        \left| \PE_{\xi_1,\xi_2}\left[ \hat g_{a_m}^{(m+1)}(X_{m+1},\theta_{m+1})
            \vert \mathcal{F}_{m} \right]
        \right|^p \right] \\
      \leq 2^p \delta^{-p} \; \sum_{m \geq n} m^{-p} \; \PE_{\xi_1,\xi_2}\left[
        \left| \hat g_{a_m}^{(m+1)}(X_{m+1},\theta_{m+1}) \right|^p \right] \\
      \leq 2^{2p-1} \; \delta^{-p} \; \sum_{m \geq n} m^{-p} \; \left\{
        \frac{c_1}{a_m^p} \; \PE_{\xi_1,\xi_2}\left[ V^{\beta p}(X_{m+1}) \right]
        + c_2 \left(\frac{N}{(1-(1-a_m)^{N})}\right)^p \right\}
  \end{multline*}
  where we used Theorem~\ref{theo:controleG} with $q = +\infty$.  Furthermore
  by Propositions~\ref{prop:ComparaisonGal}(\ref{prop:CpG1}) and
  \ref{prop:ComparaisonGal2} and the drift inequality,
  \begin{multline*}
    \PP_{\xi_1,\xi_2} \left( \sup_{m \geq n} m^{-1} (1-a_m)^{-1} \; \left|
        \PE_{\xi_1,\xi_2}\left[ \hat g_{a_m}^{(n+1)}(X_{m+1},\theta_{m+1})
          \vert
          \mathcal{F}_{m} \right] \right| \geq \delta \right) \\
    \leq \frac{2^p c_3}{\delta^{p}} \; \left\{ n^{-p} a_n^{-p}
      \PE_{\xi_1,\xi_2}[V(X_n)] + \sum_{m \geq n} m^{-p} a_m^{-p} + \sum_{m
        \geq n} m^{-p} \;
      \left(\frac{N}{(1-(1-a_m)^{N})}\right)^p \right\} \\
    \leq \frac{2^pc_3}{\delta^{p}} \; \left\{ n^{-p} a_n^{-p} \left(\xi_1(V) + n
        \bar b \right) + \bar b \sum_{m \geq n} m^{-p} a_m^{-p} + \sum_{m \geq
        n} m^{-p} \; \left(\frac{N}{(1-(1-a_m)^{N})}\right)^p \right\} \eqsp.
  \end{multline*}
  Under the stated conditions on $\{a_n, n\geq 1
  \}$, the rhs tends to zero as $n \to +\infty$.
\end{proof}

\paragraph{\tt Convergence of  Term 4.}
We prove that
\[
a_n n^{-1} (1-a_n)^{-1} \sum_{k=1}^n \PE_{\xi_1,\xi_2}\left[\hat
  g_{a_n}^{(k+1)}(X_{k+1},\theta_{k+1}) \vert \mathcal{F}_{k} \right]
\longrightarrow 0 \eqsp, \PP_{\xi_1,\xi_2}-\as
\]
provided $\{a_n^{1 \wedge [(1-\alpha-\beta)/\alpha]} \; n^{-1}, n\geq 1\}$ is
non-increasing, $\sum_n a_n^{1 \wedge [(1-\alpha-\beta)/\alpha]} \; n^{-1} <
+\infty$, and $\lim_n a_n =0$.

\begin{proof}  Choose $q \geq 1$ such that  $\beta + \alpha/q \leq 1-\alpha$. Fix $\epsilon >0$.  From Theorem~\ref{theo:controleG}, there exist  constants $C,N$ such that  for any $n\geq1$, $l\geq 0$, $(x,\theta) \in \Xset \times \Theta$,
  \[
  \left| \hat g_{a_n}^{(l)}(x,\theta) \right| \leq C \ a_n^{1/q-1}\ V^{\beta +
    \alpha /q}(x) + 4 \epsilon N (1-(1-a_n)^N)^{-1} \eqsp.
\]
Hence for $n$ large enough such that $(1-a_n) \geq 1/2$
\begin{multline*}
  \left| a_n n^{-1} (1-a_n)^{-1} \sum_{k=1}^n \PE_{\xi_1,\xi_2}\left[\hat
      g_{a_n}^{(k+1)}(X_{k+1},\theta_{k+1}) \vert \mathcal{F}_{k} \right] \right| \\
  \leq 8 a_n \epsilon N (1-(1-a_n)^N)^{-1} + 2 C \ a_n^{1/q } n^{-1} \;
  \sum_{k=1}^n \PE_{\xi_1,\xi_2} \left[ V^{\beta +
      \alpha /q}(X_{k+1}) \vert \mathcal{F}_k\right] \\
  \leq 8 a_n \epsilon N (1-(1-a_n)^N)^{-1} + 2 C \ a_n^{1/q} n^{-1} \;
  \sum_{k=1}^n V^{1-\alpha}(X_k) + 2 C\; \ a_n^{1/q} \bar b \eqsp,
\end{multline*}
where we used $\beta + \alpha/q \leq 1-\alpha$ and
Proposition~\ref{prop:ComparaisonGal}(\ref{prop:CpG1}) in the last inequality.
Since $\lim_n a_n =0$ and $\lim_n a_n \epsilon N (1-(1-a_n)^N)^{-1} =
\epsilon$, we only have to prove that $a_n^{1/q} \; n^{-1} \sum_{k=1}^n
V^{1-\alpha}(X_k)$ converges to zero $\PP_{\xi_1,\xi_2}$-\as By the Kronecker
Lemma (see e.g \cite[Section 2.6]{halletheyde80}), this amounts to prove that
$\sum_{k \geq 1} a_k^{1/q} k^{-1} \; V^{1-\alpha}(X_k)$ is finite \as This
property holds upon noting that by Proposition~\ref{prop:ComparaisonGal2} and
Proposition~\ref{prop:ComparaisonGal}(\ref{prop:CpG1})
\begin{multline*}
  \PE_{\xi_1,\xi_2} \left[ \sum_{k \geq n} a_k^{1/q} k^{-1} \;
    V^{1-\alpha}(X_k) \right] \leq a_n^{1/q} n^{-1} \; \PE_{\xi_1,\xi_2} \left[
    V(X_n)\right] +
  \sum_{k \geq n} a_k^{1/q} k^{-1} \\
  \leq a_n^{1/q} n^{-1} \; \left( \xi_1(V) + \bar b n \right)+ \sum_{k \geq n}
  a_k^{1/q} k^{-1},
\end{multline*}
and the rhs tends to zero under the stated assumptions.
\end{proof}

\subsection{Proof of Proposition~\ref{prop:SLLNUnboundedUnifCaseBounded}}
We only give the sketch of the proof since the proof is very
  similar to that of Theorem~\ref{theo:SLLNUnboundedUnifCase}.  We start with
  proving a result similar to Theorem~\ref{theo:controleG}. Since $\Dset =
  \Xset$, the sequence $\{\tau^k, k\geq 0\}$ is deterministic and $\tau^{k+1} =
  \tau^k + N +1$. By adapting the proof of Theorem~\ref{theo:controleG} ($f$ is
  bounded and $\Dset = \Xset$), we establish that for any $\epsilon>0$, there
  exists an integer $n \geq 2$ such that for any $0<a<1$, any bounded function
  $f$, $l\geq 0$, $(x,\theta) \in \Xset \times \Theta$
\[
  \left( |\bar f|_{1} \right)^{-1} \; \left| \hat g_{a}^{(l)}(x,\theta) \right|
  \leq n+ \epsilon \; \left(1-(1-a)^n \right)^{-1} \; n \eqsp.
\]
We then introduce the martingale decomposition as in the proof of
Theorem~\ref{theo:SLLNUnboundedUnifCase} and follow the same lines (with any
$p>1$).

\appendix
\section{Explicit control of convergence}
\label{app:UniformControl}
We provide sufficient conditions for the assumptions
A\ref{A-VCset}(\ref{A4rev}) and A\ref{A6}.  The technique relies on the
explicit control of convergence of a transition kernel $P$ on a general state
space $(\mathbb{T}, \mathcal{B}(\mathbb{T}))$ to its stationary distribution
$\pi$.
\begin{prop}
\label{prop:ExplicitControlCvg}
Let $P$ be a $\phi$-irreducible and aperiodic transition kernel on
$(\mathbb{T}, \mathcal{B}(\mathbb{T}))$.
\begin{enumerate}[(i)]
\item \label{block1} Assume that there exist a probability measure $\nu$ on
  $\mathbb{T}$, positive constants $\varepsilon, b,c$, a measurable set
  $\Cset$, a measurable function $V: \mathbb{T} \to [1, +\infty)$ and $0<
  \alpha \leq 1$ such that
\begin{equation}
  \label{eq:HypPropExplicitControlCvg}
  P(x,\cdot) \geq \un_\Cset(x) \; \varepsilon \  \nu(\cdot) \eqsp, \qquad \qquad PV
\leq V - c \  V^{1-\alpha} +b \  \un_\Cset \eqsp.
\end{equation}
Then $P$ possesses an invariant probability measure $\pi$ and
$\pi(V^{1-\alpha})< + \infty$.
\item \label{block2} Assume in addition that $ c \ \inf_{\Cset^c} V^{1-\alpha}
  \geq b$, $\sup_\Cset V < + \infty$ and $\nu(\Cset)>0$.  Then there exists a
  constant $C$ depending upon $\sup_\Cset V$, $\nu(\Cset)$ and $\varepsilon,
  \alpha,b,c$ such that for any $0 \leq \beta \leq 1-\alpha$ and $1 \leq \kappa
  \leq\alpha^{-1}(1-\beta)$,
\begin{equation}
  \label{eq:ConcPropExplicitControlCvg}
  (n+1)^{\kappa-1} \; \| P^n(x,\cdot) - \pi(\cdot) \|_{V^\beta} \leq C \ V^{\beta +
  \alpha \kappa}(x).
\end{equation}
\end{enumerate}
\end{prop}
\begin{proof}
  The conditions (\ref{eq:HypPropExplicitControlCvg}) imply that $V$ is
  unbounded off petite set and $P$ is recurrent. It also implies that $\{V<
  +\infty \}$ is full and absorbing: hence there exists a level set $\Dset$ of
  $V$ large enough such that $\nu(\Dset) >0$. Following the same lines as in
  the proof of Proposition~\ref{prop:TimeFiniteAS}, we prove that $\sup_\Dset
  \PE_x[\tau_\Dset] < +\infty$. The proof of (\ref{block1}) in concluded by
  \cite[Theorems 8.4.3., 10.0.1]{meynettweedie93}. The proof of (\ref{block2})
  is given in e.g.  \cite{gersendeetmoulines03} (see also
  \cite{andrieu:fort:2005,douc:moulines:soulier:2007}).

\end{proof}

When $b \leq c$, $ c \ \inf_{\Cset^c} V^{1-\alpha} \geq b$. Otherwise, it is
easy to deduce the conditions of (\ref{block2}) from conditions of the form
(\ref{block1}).

\begin{coro}
  Let $P$ be a phi-irreducible and aperiodic transition kernel on $(\mathbb{T},
  \mathcal{B}(\mathbb{T}))$. Assume that there exist positive constants $ b,c$,
  a measurable set $\Cset$, an unbounded measurable function $V: \mathbb{T} \to
  [1, +\infty)$ and $0< \alpha \leq 1$ such that $P V \leq V - c V^{1-\alpha}
  +b \un_\Cset$. Assume in addition that the level sets of $V$ are $1$-small.
  Then there exist a level set $\Dset$ of $V$, positive constants
  $\varepsilon_\Dset$, $c_\Dset$ and a probability measure $\nu_\Dset$ such
  that
\[
P(x,\cdot) \geq \un_\Dset(x) \; \varepsilon_\Dset \ \nu_\Dset(\cdot) \eqsp,
\qquad \qquad PV \leq V - c_\Dset \ V^{1-\alpha} +b \ \un_\Dset \eqsp,
\]
and $\sup_\Dset V < +\infty$, $\nu_\Dset(\Dset) >0$, and $c_\Dset \
\inf_{\Dset^c} V^{1-\alpha} \geq b$.
\end{coro}
\begin{proof}
  For any $0 < \gamma <1$, $PV \leq V - \gamma \; c \ V^{1-\alpha} +b \
  \un_{\Dset_\gamma}$ with ${\Dset_\gamma} \eqdef \{ V^{1-\alpha} \leq b [c
  (1-\gamma)]^{-1} \}$.  Hence, $\sup_{\Dset_\gamma} V < +\infty$; and for
  $\gamma$ close to $1$, we have $\gamma c \; \inf_{\Dset^c_\gamma}
  V^{1-\alpha} \geq b$. Finally, the drift condition
  (\ref{eq:HypPropExplicitControlCvg}) implies that the set $\{V < +\infty \}$
  is full and absorbing and thus the level sets $\{V \leq d \}$ are accessible
  for any $d$ large enough.
\end{proof}

The $1$-smallness assumption is usually done for convenience and is not
restrictive. In the case the level sets are petite (and thus $m$-small for some
$m \geq 1$), the explicit upper bounds get intricate and are never detailed in
the literature (at least in the polynomial case).  Nevertheless, it is a
recognized fact that the bounds derived in the case $m=1$ can be extended to the
case $m>1$.

\section{$L^p$-martingales and the Chow-Birnbaum-Marshall's inequality}
We deal with martingales and martingale arrays in the paper using the following two results.
\begin{lemma}\label{lem1martingales}
Let $\{(D_k,\F_k),\;1\leq k\geq 1\}$ be a martingale difference sequence and $M_n=\sum_{k=1}^nD_k$. For any $p>1$,
\begin{equation}
\PE\left[\left|M_n\right|^p\right]\leq Cn^{\max(p/2,1)-1}\sum_{k=1}^n \PE\left(\left|D_k\right|^p\right),\end{equation}
where $C=\left(18pq^{1/2}\right)^p$, $p^{-1}+q^{-1}=1$.
\end{lemma}
\begin{proof}
By Burkholder's inequality (\cite{halletheyde80}, Theorem 2.10) applied to the martingale $\{(M_n,\F_n),\;n\geq 1\}$, we get
\[
\PE\left(\left|M_n\right|^p\right)\leq C\PE\left[\left(\sum_{k=1}^k|D_k|^2\right)^{p/2}\right],\]
where $C=\left(18pq^{1/2}\right)^p$, $p^{-1}+q^{-1}=1$. The proof follows by noting that
\begin{equation}\label{eq:prop1}\left(\sum_{k=1}^n|D_k|^2\right)^{p/2}\leq n^{\max(p/2,1)-1}\sum_{k=1}^n\left|D_k\right|^p.\end{equation}
To prove (\ref{eq:prop1}), note that if $1<p\leq 2$, the convexity inequality $(a+b)^\alpha\leq a^\alpha+b^\alpha$ which hold true for all $a,b\geq 0$ and $0\leq \alpha\leq 1$ implies that $\left(\sum_{n=1}^n|D_k|^2\right)^{p/2}\leq \sum_{k=1}^n |D_k|^p$. If $p>2$, Holder's inequality gives $\left(\sum_{k=1}^n|D_k|^2\right)^{p/2}\leq n^{p/2-1}\left(\sum_{k=1}^n|D_k|^p\right)$.
\end{proof}

Lemma~\ref{lem:Birnbaum} can be found in  \cite{Atchade:2009} and provides a generalization to the classical Chow-Birnbaum-Marshall's
inequality.
\begin{lemma}
\label{lem:Birnbaum}
Let $\{D_{n,i},\F_{n,i},\;1\leq i\leq n\}$, $n\geq 1$ be a martingale-difference array and $\{c_n,\;n\geq 1\}$ a non-increasing sequence of positive numbers. Assume that $\F_{n,i}=\F_i$ for all $i,n$. Define
\[S_{n,k} \eqdef \sum_{i=1}^k D_{n,i},\;\; \mbox{ if }1\; \leq k\leq n \;\; \mbox{ and }\;\;\; S_{n,k} \eqdef \sum_{i=1}^n D_{n,i}+\sum_{j=n+1}^kD_{j,j},\;\;\;k>n;\]
\[R_n \eqdef \sum_{j=1}^{n-1}\left(D_{n,j}-D_{n-1,j}\right).\]
For $n\leq m\leq N$, $p\geq 1$ and $\lambda>0$
\begin{multline}2^{-p}\lambda^p\PP\left(\max_{n\leq m\leq N}c_m|M_{m,m}|>\lambda\right)\leq c_N^p\PE\left(|S_{n,N}|^p\right)+\sum_{j=n}^{N-1}\left(c_j^p-c_{j+1}^p\right)\PE\left(|S_{n,j}|^p\right) \\
+ \PE\left[\left(\sum_{j=n+1}^N c_j|R_j|\right)^p\right].\end{multline}
\end{lemma}

\section{Proofs of Section~\ref{sec:ex2}}

In the proofs, $C$ will denote a generic finite constant whose actual value
might change from one appearance to the next. The proofs below differ from
earlier works (see e.g.  \cite{gersendeetmoulines00,doucetal04}) since $q$ is
not assumed to be compactly supported.

\subsection{Proof of Lemma~\ref{driftRWM}}
\begin{lemma}
\label{lem:tool1:proof:driftRWM}
Assume D\ref{D1}-\ref{D2}. For all $x$ large enough and $|z| \leq \eta
|x|^\upsilon$, $t \mapsto V_s(x+tz)$ is twice continuously differentiable on
$[0,1]$.  There exist a constant $C < +\infty$ and a positive function
$\varepsilon$ such that $\lim_{|x| \to\infty} \varepsilon(x) = 0$, such that
for all $x$ large enough, $|z| \leq \eta |x|^\upsilon$ and $s \leq s_\star$,
\[
\sup_{t \in [0,1]} |\nabla^2 V_s(x+tz)| \leq C \; s V_s(x) |x|^{2(m-1)} \left
  (s + \varepsilon(x) \right) \eqsp.
\]
\end{lemma}
\begin{proof}
  $|x +z | \geq |x| -\eta |x|^\upsilon \geq (1-\eta) |x|^\upsilon$ so that $t
  \mapsto V_s(x+tz)$ is twice continuously differentiable on $[0,1]$ for $|x|$
  large enough. We have
  \begin{multline*}
    |\nabla^2 V_s(x+tz)| \leq s V_s(x) \ \ \frac{V_s(x+tz)}{V_s(x)} | \nabla
    \ln \pi(x+tz) \nabla \ln \pi(x+tz)^T | \cdots \\
    \left( s + \frac{|\nabla^2 \ln \pi(x+tz) |}{| \nabla \ln \pi(x+tz) \nabla
        \ln \pi(x+tz)^T |}\right)
  \end{multline*}
  Under the stated assumptions, there exists a constant $C$ such that for any
  $x$ large enough and $|z| \leq \eta |x|^\upsilon$
\[
\sup_{t \in [0,1] } \left( s + \frac{|\nabla^2 \ln \pi(x+tz) |}{| \nabla \ln
    \pi(x+tz) \nabla \ln \pi(x+tz)^T |}\right) \leq s +
\frac{D_2}{d_1^2(1-\eta)} |x|^{-m \upsilon} \eqsp,
\]
and
\[
\sup_{t \in [0,1] } | \nabla \ln \pi(x+tz) \nabla \ln \pi(x+tz)^T | \leq
|x|^{2(m-1)} D_1^2 \left( 1 -\eta |x|^{\upsilon-1} \right)^{2(m-1)} \eqsp..
\]
Finally,
\[
\sup_{t \in [0,1] ,s \leq s_\star}\left(\frac{\pi(x+tz)}{\pi(x)} \right)^{-s}
\leq 1 + s_\star D_1 \ |z| \sup_{t \in [0,1] } |x+tz|^{m-1} \sup_{t \in [0,1],s
  \leq s_\star }\left(\frac{\pi(x+tz)}{\pi(x)} \right)^{-s}
\]
which yields the desired result upon noting that $|z| |x+tz|^{m-1} \leq \eta
|x|^{\upsilon+m-1} (1-\eta |x|^{\upsilon -1})$ is arbitrarily small for $x$
large enough.
\end{proof}

We now turn to the proof of Lemma~\ref{driftRWM}. For $x\in\Xset$, define
$R(x):=\{y\in\Xset:\; \pi(y)<\pi(x)\}$ and $R(x)-x \eqdef \{y-x:\; y\in
R(x)\}$. We have:
\begin{eqnarray*}
 P_\theta V_s(x)-V_s(x)&=&\int\left(V_s(x+z)-V_s(x)\right)q_\theta(z) \ \mu_{Leb}(dz) \\
&&+ \int_{R(x)-x}\left(V(x+z)-V(x)\right)\left(\frac{\pi(x+z)}{\pi(x)}-1\right) q_\theta(z) \ \mu_{Leb}(dz) \eqsp.
\end{eqnarray*}

If $x$ remains in a compact set $\Cset$, using D\ref{D2}(\ref{D2a}) and the
continuity of $x \mapsto V_s(x)$, we have $V_s(x+z)\leq C(1+ \exp(s D_0
|z|^m))$.  It follows that
\[
\sup_{\theta \in \Theta} \sup_{x \in \Cset} \{ P_\theta V_s(x) - V_s(x) \} \leq
C \ \sup_{\theta \in \Theta} \int_{R(x)-x} (1+ \exp(s D_0 |z|^m)) \ q_\theta(z) \
\mu_{Leb}(dz) < +\infty \eqsp.
\]

More generally, let $x$ large enough. Define $l(x) \eqdef \log\pi(x)$,
$R_V(x,z)\eqdef V_s(x+z)-V_s(x)+ s V_s(x) \pscal{z}{\nabla l(x)}$,
$R_\pi(x,z)\eqdef \pi(x+z)(\pi(x))^{-1}-1-\pscal{z}{\nabla l(x)}$. Using the
fact that the mean of $q_\theta$ is zero, we can write: $P_\theta V_s(x) -
V_s(x)=I_1(x,\theta,s)+I_2(x,\theta,s)+I_3(x,\theta,s) $ where
\[I_1(x,\theta,s)\eqdef - s V_s(x) \int_{R(x)-x} \pscal{z}{\nabla l(x)}^2 \ q_\theta(z) \ \mu_{Leb}(dz)  \eqsp,\]
\[I_2(x,\theta,s)\eqdef \int R_V(x,z) \; q_\theta(z) \ \mu_{Leb}(dz)+\int_{R(x)-x}R_V(x,z)\left(\frac{\pi(x+z)}{\pi(x)}-1\right) \ q_\theta(z) \ \mu_{Leb}(dz) \eqsp,\]
and
\[I_3(x,\theta,s) \eqdef - s V_s(x) \ \int_{R(x)-x} R_\pi(x,z)\pscal{z}{\nabla l(x)} \ q_\theta(z) \ \mu_{Leb}(dz) \eqsp.\]
\subsubsection{First  term}
It follows from \cite[Lemma B.3. and proof of Proposition
3]{gersendeetmoulines00} that, under D\ref{D2}(\ref{D2z}), there exists $b>0$,
such that for all $\theta \in \Theta$,
\[
\int_{R(x)-x} \pscal{z}{\nabla l(x)}^2 \ q_\theta(z) \ \mu_{Leb}(dz) \geq b \;
|\nabla l(x) |^2 \eqsp.
\]
Hence, $\sup_{\theta \in \Theta} I_1(x,\theta,s) \leq -s \; V_s(x) \ b \; d_1^2
|x|^{2(m-1)}$.

\subsubsection{Second term}
For $z\in R(x)-x$, $\pi(x+z)<\pi(x)$. Therefore $|I_2(x,\theta,s)|\leq 2\int
|R_V(x,z)|q_\theta(z) \ \mu_{Leb}(dz)$.  By
Lemma~\ref{lem:tool1:proof:driftRWM}, there exists $C< + \infty$ - independent
of $s$ for $s \leq s_\star$- such that for any $|z| \leq \eta |x|^\upsilon$,
\[
|R_V(x,z) |\leq C \; s \; V_s(x) \ |x|^{2(m-1)} \ |z|^2 \ \left(s +
  \varepsilon(x) \right)\eqsp.
\]
This implies that there exists a constant $C< +\infty$ - independent of $s$ for
$s \leq s_\star$ - such that
\begin{multline*}
  \int |R_V(x,z)|q_\theta(z) \ \mu_{Leb}(dz) \leq C \; s \; V_s(x) \
  |x|^{2(m-1)} \ \left(s +
    \varepsilon(x) \right) \ \int |z|^2 q_\theta(z) \mu_{Leb}(dz) \\
  + V_s(x) \; \int_{\{z, |z| \geq \eta |x|^\upsilon \}}
  \frac{V_s(x+z)}{V_s(x)} \  q_\theta(z) \mu_{Leb}(dz) \\
  + C \; V_s(x) \; |x|^{m-1} \int_{\{z, |z| \geq \eta |x|^\upsilon \}} |z| \;
  q_\theta(z) \mu_{Leb}(dz) \eqsp.
\end{multline*}
There exists a constant $C$ such that for $\theta \in \Theta$ and $s \leq
s_\star$, the first term in the rhs is upper bounded by $C \; s \; V_s(x) \
|x|^{2(m-1)} \left(s + \varepsilon(x) \right)$. Under D\ref{D3}, the second
term is upper bounded by $V_s(x) \; |x|^{2(m-1)} \; \varepsilon(x) $ with
$\lim_{|x| \to +\infty} \varepsilon(x) = 0$ uniformly in $\theta$ for $\theta
\in \Theta$, and in $s$ for $s \leq s_\star$.  Since $q_\theta$ is a
multivariate Gaussian distribution, there exists $\lambda_\star>0$ such that
$\sup_{\theta \in\Theta} \int \exp(\lambda_\star |z|^2) q_\theta(z)
\mu_{Leb}(dz)< +\infty$. Under D\ref{D3}, the third term is upper bounded by $C
\; V_s(x) \; |x|^{2(m-1)} \; \exp(-\lambda \eta^2 |x|^{2 \upsilon})$ for some
$\lambda \in (0,\lambda_\star)$, uniformly in $\theta$ for $\theta \in \Theta$,
and in $s$ for $s \leq s_\star$. Hence, we proved that there exists $C_\star<
\infty$ such that for any $s \leq s_\star$,
\[
\sup_{\theta \in \Theta} |I_2(x,\theta,s)|\leq C_\star \; V_s(x) \;
|x|^{2(m-1)} \; \left(s^2+ \varepsilon(x) \right) \eqsp,
\]
for a positive function $\varepsilon$ independent of $s$ and such that
$\lim_{|x| \to +\infty} \varepsilon(x) = 0$.

\subsubsection{Third term} Following the same lines as in the control of $I_2(x,\theta,s)$,  it may be proved that
\begin{multline*}
  I_3(x,\theta,s) \leq s V_s(x) D_1 |x|^{m-1} \int_{\{z, |z| \geq \eta
    |x|^\upsilon \}} |z | \left( 1 + D_1 |z| |x|^{m-1} \right) q_\theta(z) \mu_{Leb}(dz) \\
  + C \ V_s(x) |x|^{3(m-1)} \; \int_{\{z, |z| \leq \eta |x|^\upsilon \}} |z|^3
  \ q_\theta(z) \mu_{Leb}(dz) \leq C \ V_s(x) |x|^{2(m-1)} \varepsilon(x)
\end{multline*}
for a positive function $\varepsilon$ independent of $s,\theta$ and such that
$\lim_{|x| \to +\infty} \varepsilon(x) = 0$.

\subsubsection{Conclusion}
Let $\alpha \in (0,1)$.  By combining the above calculations, we prove that by
choosing $s$ small enough such that $c_\star \eqdef b d_1^2 - C_\star s>0$, we
have
\begin{align}
  \sup_{\theta \in \Theta} P_\theta V_s(x) & \leq V_s(x)  - c_\star V_s(x)
  |x|^{2(m-1)} + b_\star \un_\Cset(x) \label{eq:drift:sous-geom}  \\
  & \leq  V_s(x) - 0.5 c_\star V_s^{1-\alpha}(x) + b_\star \un_\Cset(x)
\end{align}
for a compact set $\Cset$. This proves A\ref{Adrift}(ii) and A\ref{A5}.
A\ref{A6} follows from the results of Appendix~\ref{app:UniformControl}.
A\ref{Adrift}(iii) and A\ref{A2} follow from Lemma~\ref{lem:example:smallset}.

\subsection{Proof of Lemma~\ref{ex:lem:HypB}}
An easy modification in the proof of \cite[Proposition 11]{andrieuetal06} (to
adjust for the difference in the drift function) shows that
$D(\theta,\theta')\leq 2\int_\Xset |q_{e^c\Sigma}(x)-q_{e^{c'}\Sigma'}(x)|
\mu_{Leb}(dx)$. We then apply \cite[Lemma 12]{andrieuetal06} to obtain that
$D(\theta,\theta')\leq C \, |e^c \Sigma-e^{c'}\Sigma'|_\s$ where $C$ is a
finite constant depending upon the compact $\Theta$.  Hereafter, $C$ is finite
and its value may change upon each appearance.  For any $l,n\geq 0$,
$\epsilon>0$, $x \in \rset^p$ and $\theta\in\Theta$, we have
\begin{eqnarray*}
\PP^{(l)}_{x,\theta} \left(D(\theta_n,\theta_{n+1})\geq \epsilon\right) &\leq& \epsilon^{-1}\PE^{(l)}_{x,\theta}\left[D(\theta_n, \theta_{n+1})\right]\\
&\leq&   C \, \PE^{(l)}_{x,\theta}\left[ |c_{n+1}-c_n| + |\Sigma_{n+1} - \Sigma_n|_\s \right]\\
&\leq& C  \, (l+n+1)^{-1}\left(1+\PE^{(l)}_{x,\theta}\left[|X_{n+1}|^2\right]+\sqrt{\PE^{(l)}_{x,\theta}\left[|X_{n+1}|^2\right]}\right) \eqsp.
\end{eqnarray*}
D\ref{D2}(\ref{D2a}) implies that we can find $C<\infty$ such that $|x|^2\leq C
\; \phi(V_s(x))$ for all $x\in\Xset$ where $\phi(t) = [\ln t]^{2/m}$. From the
drift condition (Lemma~\ref{driftRWM}),
Proposition~\ref{prop:ComparaisonGal}(\ref{prop:CpG1}) and the concavity of
$\phi$, we deduce that there exists $C$ such that
$\PE^{(l)}_{x,\theta}\left[|X_n|^2\right]\leq C\; [\ln V_s(x) ]^{2/m} \; [\ln
n]^{2/m}$.  We conclude that for any probability $\xi_1$ such that $\xi_1([\ln
V_s]^{2/m}) < +\infty$, $\lim_n \PP_{\xi_1,\xi_2}
\left(D(\theta_n,\theta_{n+1}) \geq \epsilon\right)=0$ and for any level set
$\mathcal{D}$ of $V_s$,
\[
\lim_{n\to\infty}\sup_{l\geq 0}\sup_{\mathcal{D}\times
  \Theta}\PP^{(l)}_{x,\theta} \left(D(\theta_n,\theta_{n+1}) \geq
  \epsilon\right)=0 \eqsp.\]

\vspace{2cm}

{\bf Acknowledgment:} We would like to thank Michael Woodroofe for helpful
discussions on the resolvent approach to limit theorems and Prof. Pierre
Priouret and Christophe Andrieu for helpful discussions. We also thank M.
Vihola for helpful comments. 


\begin{thebibliography}{34}
\expandafter\ifx\csname natexlab\endcsname\relax\def\natexlab#1{#1}\fi
\expandafter\ifx\csname url\endcsname\relax
  \def\url#1{\texttt{#1}}\fi
\expandafter\ifx\csname urlprefix\endcsname\relax\def\urlprefix{URL }\fi

\bibitem[{Andrieu and Atchade(2007)}]{andrieuetatchade05}
\textsc{Andrieu, C.} and \textsc{Atchade, Y.~F.} (2007).
\newblock On the efficiency of adaptive {MCMC} algorithms.
\newblock \textit{Electronic Communications in Probability} \textbf{12}
  336--349.

\bibitem[{Andrieu and Fort(2005)}]{andrieu:fort:2005}
\textsc{Andrieu, C.} and \textsc{Fort, G.} (2005).
\newblock Explicit control of subgeometric ergodicity.
\newblock Tech. rep., University of Bristol, 05:17.
\newblock Available from http://www.tsi.enst.fr/$\sim$gfort/biblio.html.

\bibitem[{Andrieu and Moulines(2006)}]{andrieuetal06}
\textsc{Andrieu, C.} and \textsc{Moulines, {\'E}.} (2006).
\newblock On the ergodicity properties of some adaptive {MCMC} algorithms.
\newblock \textit{Ann. Appl. Probab.} \textbf{16} 1462--1505.

\bibitem[{Andrieu and Robert(2001)}]{andrieuetrobert02}
\textsc{Andrieu, C.} and \textsc{Robert, C.~P.} (2001).
\newblock Controlled {MCMC} for optimal sampling.
\newblock \textit{Technical report, Universit\'e Paris Dauphine, Ceremade 0125}
  .

\bibitem[{Andrieu and Tadic(2008)}]{andrieu:vlad:2008}
\textsc{Andrieu, C.} and \textsc{Tadic, V.} (2008).
\newblock General result for the stability of controlled {MCMC}.
\newblock Tech. rep., Bristol University.
\newblock (personal communication).

\bibitem[{Atchade and Fort(2008)}]{atchade:fort:2008b}
\textsc{Atchade, Y.} and \textsc{Fort, G.} (2008).
\newblock Limit theorems for some adaptive {MCMC} algorithms with subgeometric
  kernels ({II}).
\newblock Tech. rep., Work in progress.

\bibitem[{Atchade(2006)}]{atchade05}
\textsc{Atchade, Y.~F.} (2006).
\newblock An adaptive version for the {M}etropolis adjusted {L}angevin
  algorithm with a truncated drift.
\newblock \textit{Methodol Comput Appl Probab} \textbf{8} 235--254.

\bibitem[{Atchade(2009)}]{Atchade:2009}
\textsc{Atchade, Y.~F.} (2009).
\newblock A strong law of large numbers for martingale arrays.
\newblock Tech. rep., Univ. of Michigan.
\newblock Available at http://www.stat.lsa.umich.edu/~yvesa/.

\bibitem[{Atchade and Rosenthal(2005)}]{atchadeetrosenthal03}
\textsc{Atchade, Y.~F.} and \textsc{Rosenthal, J.~S.} (2005).
\newblock On adaptive {M}arkov chain {M}onte {C}arlo algorithm.
\newblock \textit{Bernoulli} \textbf{11} 815--828.

\bibitem[{Bai(2008)}]{Bai:2008}
\textsc{Bai, Y.} (2008).
\newblock The simultaneous drift conditions for {A}daptive {M}arkov {C}hain
  {M}onte {C}arlo algorithms.
\newblock Tech. rep., Univ. of Toronto.
\newblock (personal communication).

\bibitem[{Benveniste et~al.(1987)Benveniste, Métivier and
  Priouret}]{benveniste:metivier:priouret:1987}
\textsc{Benveniste, A.}, \textsc{Métivier, M.} and \textsc{Priouret, P.}
  (1987).
\newblock \textit{Adaptive algorithms and {S}tochastic {A}pproximations}.
\newblock Springer-Verlag.

\bibitem[{Chen et~al.(1988)Chen, Guo and Gao}]{chen:gua:gao:1988}
\textsc{Chen, H.}, \textsc{Guo, L.} and \textsc{Gao, A.} (1988).
\newblock Convergence and robustness of the {R}obbins-{M}onro algorithm
  truncated at randomly varying bounds.
\newblock \textit{Stochastic Process. Appl.} \textbf{27} 217--231.

\bibitem[{Chen and Zhu(1986)}]{chen:zhu:1986}
\textsc{Chen, H.} and \textsc{Zhu, Y.} (1986).
\newblock Stochastic approximation procedures with randomly varying
  truncations.
\newblock \textit{Sci. Sinica. Ser. A} \textbf{29} 914--926.

\bibitem[{Douc et~al.(2004)Douc, Fort, Moulines and Soulier}]{doucetal04}
\textsc{Douc, R.}, \textsc{Fort, G.}, \textsc{Moulines, E.} and
  \textsc{Soulier, P.} (2004).
\newblock Practical drift conditions for subgeometric rates of convergence.
\newblock \textit{Ann. Appl. Probab.} \textbf{14} 1353--1377.

\bibitem[{Douc et~al.(2007)Douc, Moulines and
  Soulier}]{douc:moulines:soulier:2007}
\textsc{Douc, R.}, \textsc{Moulines, E.} and \textsc{Soulier, P.} (2007).
\newblock Computable convergence rates for sub-geometric ergodic {M}arkov
  chains.
\newblock \textit{Bernoulli} \textbf{13} 831--848.

\bibitem[{Fort and Moulines(2000)}]{gersendeetmoulines00}
\textsc{Fort, G.} and \textsc{Moulines, E.} (2000).
\newblock {$V$}-subgeometric ergodicity for a {H}astings-{M}etropolis
  algorithm.
\newblock \textit{Statist. Probab. Lett.} \textbf{49} 401--410.

\bibitem[{Fort and Moulines(2003)}]{gersendeetmoulines03}
\textsc{Fort, G.} and \textsc{Moulines, E.} (2003).
\newblock Polynomial ergodicity of {M}arkov transition kernels.
\newblock \textit{Stochastic Process. Appl.} \textbf{103} 57--99.

\bibitem[{Gilks et~al.(1998)Gilks, Roberts and Sahu}]{gilksetal98}
\textsc{Gilks, W.~R.}, \textsc{Roberts, G.~O.} and \textsc{Sahu, S.~K.} (1998).
\newblock Adaptive {M}arkov chain {M}onte {C}arlo through regeneration.
\newblock \textit{J. Amer. Statist. Assoc.} \textbf{93} 1045--1054.

\bibitem[{Haario et~al.(2001)Haario, Saksman and Tamminen}]{haarioetal00}
\textsc{Haario, H.}, \textsc{Saksman, E.} and \textsc{Tamminen, J.} (2001).
\newblock An adaptive {M}etropolis algorithm.
\newblock \textit{Bernoulli} \textbf{7} 223--242.

\bibitem[{Hall and Heyde(1980)}]{halletheyde80}
\textsc{Hall, P.} and \textsc{Heyde, C.~C.} (1980).
\newblock \textit{Martingale Limit theory and its application}.
\newblock Academic Press, New York.

\bibitem[{Hastings(1970)}]{hastings:1970}
\textsc{Hastings, W.~K.} (1970).
\newblock {M}onte {C}arlo sampling methods using {M}arkov chains and their
  application \textbf{57} 97--109.

\bibitem[{Holden(1998)}]{holden98}
\textsc{Holden, L.} (1998).
\newblock Adaptive chains.
\newblock \textit{Technical Report} .

\bibitem[{Jarner and Hansen(2000)}]{jarnerethansen98}
\textsc{Jarner, S.~F.} and \textsc{Hansen, E.} (2000).
\newblock Geometric ergodicity of {M}etropolis algorithms.
\newblock \textit{Sto. Proc. Appl.} \textbf{85} 341--361.

\bibitem[{Jarner and Roberts(2002)}]{jarneretroberts02}
\textsc{Jarner, S.~F.} and \textsc{Roberts, G.~O.} (2002).
\newblock Polynomial convergence rates of {M}arkov chains.
\newblock \textit{Ann. Appl. Probab.} \textbf{12} 224--247.

\bibitem[{Maxwell and Woodroofe(2000)}]{mw00}
\textsc{Maxwell, M.} and \textsc{Woodroofe, M.} (2000).
\newblock Central limit theorems for additive functional of {M}arkov chains.
\newblock \textit{Annals of Probability} \textbf{28} 713--724.

\bibitem[{Merlevede et~al.(2006)Merlevede, Peligrad and Utev}]{merlevedeetal06}
\textsc{Merlevede, F.}, \textsc{Peligrad, M.} and \textsc{Utev, S.} (2006).
\newblock Recent advances in invariances principles for stationary sequences.
\newblock \textit{Probability surveys} \textbf{3} 1--36.

\bibitem[{Metropolis et~al.(1953)Metropolis, Rosenbluth, Rosenbluth, Teller and
  Teller}]{metropolis:1953}
\textsc{Metropolis, N.}, \textsc{Rosenbluth, A.~W.}, \textsc{Rosenbluth,
  M.~N.}, \textsc{Teller, A.~H.} and \textsc{Teller, E.} (1953).
\newblock Equations of state calculations by fast computing machines.
\newblock \textit{J. Chem. Phys.} \textbf{21} 1087--1092.

\bibitem[{Meyn and Tweedie(1993)}]{meynettweedie93}
\textsc{Meyn, S.~P.} and \textsc{Tweedie, R.~L.} (1993).
\newblock \textit{Markov chains and stochastic stability}.
\newblock Springer-Verlag London Ltd., London.

\bibitem[{Roberts and Rosenthal(2004)}]{roberts:rosenthal:2004}
\textsc{Roberts, G.} and \textsc{Rosenthal, J.} (2004).
\newblock General state space {M}arkov chains and {MCMC} algorithms.
\newblock \textit{Prob. Surveys} \textbf{1} 20--71.

\bibitem[{Roberts and Rosenthal(2001)}]{robertsetrosenthal01}
\textsc{Roberts, G.~O.} and \textsc{Rosenthal, J.~S.} (2001).
\newblock Optimal scaling of various {M}etropolis-{H}astings algorithms.
\newblock \textit{Statistical Science} \textbf{16}.

\bibitem[{Roberts and Rosenthal(2007)}]{rosenthaletroberts05}
\textsc{Roberts, G.~O.} and \textsc{Rosenthal, J.~S.} (2007).
\newblock Coupling and ergodicity of adaptive {MCMC}.
\newblock \textit{Journal of Applied Probablity} \textbf{44} 458--475.

\bibitem[{Roberts and Tweedie(1996)}]{robertsettweedie96}
\textsc{Roberts, G.~O.} and \textsc{Tweedie, R.~L.} (1996).
\newblock Geometric convergence and central limit theorems for multidimensional
  {H}astings and {M}etropolis algorithms.
\newblock \textit{Biometrika} \textbf{83} 95--110.

\bibitem[{Winkler(2003)}]{winkler03}
\textsc{Winkler, G.} (2003).
\newblock \textit{Image analysis, random fields and {M}arkov chain {M}onte
  {C}arlo methods}, vol.~27 of \textit{Applications of Mathematics (New York)}.
\newblock 2nd ed. Springer-Verlag, Berlin.

\bibitem[{Yang(2007)}]{yang:2007}
\textsc{Yang, C.} (2007).
\newblock Recurrent and ergodic properties of {A}daptive {MCMC}.
\newblock Tech. rep., University of Toronto, Canada.
\newblock Available from http://probability.ca/jeff/ftpdir/chao3.pdf.

\end{thebibliography}

\end{document}